\documentclass[12pt]{amsart}

\usepackage[utf8]{inputenc}
\usepackage[T1]{fontenc}
\usepackage{amsmath, mathtools, amsthm}
\usepackage{amssymb,url,xspace}
\usepackage[mathscr]{eucal}
\usepackage[dvipsnames]{xcolor}
\usepackage{dsfont} 
\usepackage{cases}
\usepackage{tikz}
\usepackage{empheq}
\usetikzlibrary{arrows.meta}
\usepackage{xfrac}
\usepackage{color}
\usepackage{empheq}
\usepackage[hidelinks]{hyperref}
\usepackage[capitalise]{cleveref}
\hypersetup{
	colorlinks,
	citecolor=red!70!black,
	filecolor=black,
	linkcolor=blue!85!black,
	urlcolor=black
}
\usepackage{enumitem}

\newcommand*{\Scale}[2][4]{\scalebox{#1}{$#2$}}

\calclayout

\newtheorem{theorem}{Theorem}[section]
\newtheorem{definition}[theorem]{Definition}
\newtheorem{remark}[theorem]{Remark}
\newtheorem{lemma}[theorem]{Lemma}
\newtheorem{proposition}[theorem]{Proposition}

\title[Feedback stabilization for the intrinsic GEB]{Boundary feedback stabilization for the intrinsic geometrically exact beam model}

\author{Charlotte Rodriguez}
\thanks{
  \textbf{Funding:} This project has received funding from the European Union’s Horizon 2020 research and innovation programme under the Marie Sklodowska-Curie grant agreement No.765579-ConFlex.}

\address {Charlotte Rodriguez \newline \indent
    {Department Mathematik, Lehrstuhl f\"ur Angewandte Mathematik, \newline \indent
    Friedrich-Alexander-Universit\"at Erlangen-N\"urnberg},
    \newline \indent
    {Cauerstr. 11, 91058 Erlangen, Germany}
  }
\email{\texttt{charlotte.rodriguez@fau.de}}
  
\date{\today}

\author{G\"unter Leugering}

\address {G\"unter Leugering \newline \indent
    {Department Mathematik, Lehrstuhl f\"ur Angewandte Mathematik, \newline \indent
    Friedrich-Alexander-Universit\"at Erlangen-N\"urnberg},
    \newline \indent
    {Cauerstr. 11, 91058 Erlangen, Germany}
  }
\email{\texttt{guenter.leugering@fau.de}}

\begin{document}

\maketitle

\begin{abstract}
In this work we address the problem of boundary feedback stabilization for a geometrically exact shearable beam, allowing for large deflections and rotations and small strains.
The corresponding mathematical model may be written in terms of displacements and rotations (GEB), or intrinsic variables (IGEB). A nonlinear transformation relates both models, allowing to take advantage of the fact that the latter model is a one-dimensional first-order semilinear hyperbolic system, and deduce stability properties for both models.
By applying boundary feedback controls at one end of the beam, while the other end is clamped, we show that the zero steady state of IGEB is locally exponentially stable for the $H^1$ and $H^2$ norms. The proof rests on the construction of a Lyapunov function, where the theory of Coron \& Bastin '16 plays a crucial role. The major difficulty in applying this theory stems from the complicated nature of the nonlinearity and lower order term where no smallness arguments apply. Using the relationship between both models, we deduce the existence of a unique solution to the GEB model, and properties of this solution as time goes to $+\infty$.
\end{abstract}

\tableofcontents

\addtocontents{toc}{\protect\setcounter{tocdepth}{1}}

\section{Introduction and main results}

Beam models describing the three-dimen\-sional motion of thin elastic bodies undergoing large deflections and rotations  have found many applications in civil, mechanical and aerospace engineering. Depending on the assumptions made on the beam (material law, motion magnitude, shearing)
there are various PDE models for flexible beams, e.g. the  Euler-Bernoulli, Rayleigh and Timoshenko beam equations accounting for small displacements and strains. However, when deflections and rotations are not small compared to the overall dimensions of the body, and this is the case for modern highly flexible light weight structures, a geometrically nonlinear model is needed. Examples include robotic arms \cite{grazioso2018robot} as well as flexible aircraft wings \cite{palacios10aircraft} or wind turbine blades \cite{wang2014windturbine} designed to be lighter and slender to improve aerodynamic efficiency. 

The \emph{geometrically exact beam} (GEB) model (see System \eqref{eq:GEB}) and the \emph{intrinsic geometrically exact beam} (IGEB) model (see System \eqref{eq:intrinsic_original}) discussed in this article are such geometrically nonlinear models
describing the motion of a beam in $\mathbb{R}^3$. They take into account shearing without warping: cross sections remain plane, do not change of shape, but may rotate independently from the motion of the centerline. The beam may undergo large displacements of its centerline and large rotations of its cross sections. Both systems are one-dimensional. The former, a second-order nonlinear system of six equations, originates from the works of Reissner \cite{reissner1981finite} and Simo \cite{simo1985finite}.
The latter, a first-order semilinear hyperbolic system of twelve equations, originates from the work of Hodges \cite{hodges1990, hodges2003geometrically}, and is of its own interest in aeroelastic modelling and engineering, see \cite{artola2019, palacios10aircraft}
and references therein. 

The adjective \textit{intrinsic} indicates that the equations make no reference to displacements or rotations variables, and instead involve only so-called \emph{intrinsic variables} (velocities and strains). Indeed, the IGEB model has velocities and strains expressed in a \emph{body attached coordinate system} as unknown states\footnote{The unknown states of the IGEB model may also be taken as velocities, and forces and moments (instead of strains), using a change of variable; see \cref{rem:mass_flex}.}, while the unknown states of the GEB model are displacements and rotations expressed in a \emph{fixed coordinate system}.
In fact, a nonlinear transformation $\mathcal{N}$, defined in \eqref{eq:transfo}, allows to express the unknowns of the IGEB model as a function of the unknowns of the GEB model, thus directly deriving the IGEB model from the GEB model (see \cite[Sec. 2.3.2]{weiss99}). This transformation, going from a nonlinear system to a semilinear and hyperbolic system, allows us to take advantage of the simpler structure of the IGEB model to study the geometrically exact beam, while using this relationship between both models to also deduce results on the GEB model.

Whereas the notions of well-posedness and stabilization of classical beam models such as Euler-Bernoulli and Timoshenko models have been extensively studied in mathematical literature (see for instance
\cite{alabau2007, tucsnak2000, guo2004, hegarty12,kimrenardy87, morgul91, krstic2009, xu2005}),
the GEB and IGEB models have, to the best knowledge of the authors, not been addressed in the literature. The main result of our work is, thus, a novel contribution in this direction.
We investigate the local exponential stabilization problem for the IGEB model, in the case a vibrating beam with one clamped end and the other being under feedback control. 
Using the transformation $\mathcal{N}$, we also deduce the existence of a unique solution, global in time, to the GEB model, and properties of this solution as time goes to infinity.

Among the methods commonly used to study stability, we opt for using a so-called \emph{quadratic Lyapunov functional}.
Our approach relies on the fact that the IGEB model is a one-dimensional first-order hyperbolic system.
For such systems, Bastin \& Coron \cite{Coron_2016_stab, bastin2017exponential} have systematized the search of quadratic Lyapunov functionals, giving sufficient criteria for their existence. More precisely, it is sufficient to find a matrix-valued function $Q$ fulfilling matrix inequalities that involve both the coefficients appearing in the equations and the boundary conditions (and consequently, the feedback control).
In view of this, our task consists in finding an appropriate feedback control and studying the IGEB model (energy of the beam, coefficients), in order to prove that there exists such a function $Q$.

Up to the best of our knowledge, the stabilization result presented here is the first result which makes use of the technique of \cite{Coron_2016_stab, bastin2017exponential} in the context of precise mechanical models for beams, such as the IGEB model. These results have been applied to chemotaxis models or the Saint-Venant equations for instance.
An additional feature in our system is that the nonlinear term cannot be made arbitrarily small.

Here\footnote{The IGEB model may be given for more general beams where the only assumption made is the thinness of the beam and that the material is linear-elastic, see \cite{artola2019, hodges2003geometrically, palacios2017invariant}. 
For simplicity, we work under more restrictive assumptions, but the more general case is of practical interest (for instance, for u-shaped plane wings).} we consider a slender beam made of an isotropic linear-elastic material, with constant geometrical and material parameters (density $\rho>0$, cross section area $a>0$, shear modulus $G>0$, Young modulus $E>0$, area moments of inertia $I_2, I_3>0$, shear correction factors $k_2>0, \ k_3 >0$, and the factor $k_1>0$ that corrects the polar moment of area), and such that sectional principal axes are aligned with the body-attached basis (see \cref{sec:unknowns_coef}).

\subsection{The models}
\label{sec:model}
Consider a beam of length $\ell>0$. We now describe the GEB and IGEB models, and the transformation relating them.

\subsubsection*{The GEB model}
The unknown states of the GEB model are the \emph{position of the centerline} and \emph{rotation matrix} $(\mathbf{p},\mathbf{R}) \colon [0, \ell]\times [0, T] \rightarrow \mathbb{R}^3 \times \mathrm{SO}(3)$. Here, $\mathrm{SO}(3)$ denotes the special orthogonal group, namely, the set of unitary real matrices of size $3$, with determinant equal to $1$, also called \emph{rotation} matrices.
With the geometrical and material restrictions described above (following \cite{strohm_dissert}), the dynamics of these unknowns are given by the system\footnote{\label{foot:cross_prod}
Here, $u \times \zeta$ denotes the cross product between any $u, \zeta \in \mathbb{R}^3$, and we shall also write $\widehat{u} \,\zeta = u \times \zeta$, meaning that $\widehat{u}$ is the skew-symmetric matrix 
\begin{align*}
\widehat{u} = \Scale[1]{ \begin{bmatrix}
0 & -u_3 & u_2 \\
u_3 & 0 & -u_1 \\
-u_2 & u_1 & 0
\end{bmatrix}};
\end{align*}
and for any skew-symmetric $\mathbf{u} \in \mathbb{R}^{3 \times 3}$, the vector $\mathrm{vec}(\mathbf{u}) \in \mathbb{R}^3$ is such that $\mathbf{u} = \widehat{\mathrm{vec}(\mathbf{u})}$.
}
\begin{align} \label{eq:GEB}
\begin{cases}
\rho a \partial_{t}^2\mathbf{p} = \partial_x(\mathbf{R} S_1\Gamma)  & \text{in }(0,\ell) \times (0,T)\\
\rho \partial_t (\mathbf{R} J W) = \partial_x(\mathbf{R} S_2 \Upsilon) + (\partial_x \mathbf{p}) \times (\mathbf{R}S_1\Gamma) & \text{in }(0,\ell) \times (0,T)\\
\mathbf{p}(\ell, \cdot) = h^\mathbf{p}, \quad \mathbf{R}(\ell, \cdot) = h^\mathbf{R}  &\text{for } t \in  (0,T)\\
-\mathbf{R}(0, \cdot) S_1 \Gamma(0, \cdot) = h_1(\cdot), \quad - \mathbf{R}(0, \cdot)S_2 \Upsilon(0, \cdot) = h_2(\cdot) &\text{for } t \in  (0,T)\\
\mathbf{p}(\cdot,0)=\mathbf{p}^0(\cdot), \quad \partial_t \mathbf{p}(\cdot,0)=v^0(\cdot) &\text{for } x \in (0,\ell)\\
\mathbf{R}(\cdot,0) = \mathbf{R}^0(\cdot), \quad \mathbf{R}(\cdot, 0)W(\cdot,0) = w^0(\cdot) &\text{for } x \in (0,\ell),
\end{cases}
\end{align}
where the functions $V, W, \Gamma, \Upsilon \colon [0, \ell] \times [0, T] \rightarrow \mathbb{R}^3$ depend nonlinearly on the unknowns $\mathbf{p}, \mathbf{R}$: they are the \emph{linear velocity}, \emph{angular velocity}, \emph{translational strain} and \emph{rotational strain} (or \emph{curvature}) of the beam respectively, defined by (see \cref{foot:cross_prod})
\begin{align} \label{eq:def_VWGU}
\begin{aligned}
V &= \mathbf{R}^\intercal \partial_t \mathbf{p}, &  W &= \mathrm{vec}(\mathbf{R}^\intercal \partial_t \mathbf{R}),\\
\Gamma  &= \mathbf{R} ^\intercal \partial_x \mathbf{p}  - e_1,&  \Upsilon  &= \mathrm{vec}\left(\mathbf{R} ^\intercal \partial_x \mathbf{R}  - R^\intercal \tfrac{\mathrm{d}}{\mathrm{d}x} R\right),
\end{aligned}
\end{align}
where $e_1 = (1, 0, 0)^\intercal$, and $R \in H^3(0, \ell; \mathrm{SO}(3))$ is a given function describing the beam before deformation.
Details on these functions and the unknown states are provided in \cref{sec:unknowns_coef}.
The beam is clamped at $x=\ell$, as seen by the Dirichlet boundary conditions in which $h^\mathbf{p} \in \mathbb{R}^3$ and $h^\mathbf{R} \in \mathbb{R}^{3 \times 3}$ are constant. 
At the other end, $x=0$, Neumann boundary controls $h_1(t)$, $h_2(t)$ $\in \mathbb{R}^3$ in feedback form are applied:
\begin{align} \label{eq:feedback_geb}
h_1(t) = - \mu_1 \mathbf{R}(0, t) V(0, t), \qquad h_2(t) = - \mu_2 \mathbf{R}(0, t)W(0, t),
\end{align}
where the positive constants $\mu_1, \mu_2>0$ are feedback parameters.
The initial data are $v^0, w^0 \in C^2([0, \ell]; \mathbb{R}^3)$, $\mathbf{p}^0 \in H^3(0, \ell; \mathbb{R}^3)$ and $\mathbf{R}^0 \in H^3(0, \ell; \mathrm{SO}(3))$.
The positive definite constant matrices $J,S_1, S_2\in \mathbb{R}^{3 \times 3}$ are defined in \eqref{eq:def_J_S1_S2}. 

\begin{remark}[External forces and moments] \label{rem:freebeam_GEB}
In this work we consider the case of a freely vibrating beam, meaning that the external forces $\bar{\phi}$ and moments $\bar{\psi}$ have been set to zero.
In the general case, with such forces and moments the governing equations of \eqref{eq:GEB} would read
\begin{align} \label{eq:geb_with_extforces}
\begin{aligned}
\rho a \partial_{t}^2\mathbf{p} &= \partial_x(\mathbf{R} S_1\Gamma) + \bar{\phi}\\
\rho \partial_t (\mathbf{R} J W) &= \partial_x(\mathbf{R} S_2 \Upsilon) + (\partial_x \mathbf{p}) \times (\mathbf{R}S_1\Gamma) + \bar{\psi}.
\end{aligned}
\end{align}
See also \cref{rem:extensions} \ref{rem:freely_vibrating}.
\end{remark}


\subsubsection*{The transformation}
Departing from System \eqref{eq:GEB}, one obtains the IGEB model described below, by applying the nonlinear transformation $\mathcal{N}$
\begin{align} \label{eq:transfo}
\mathcal{N} \colon \quad (\mathbf{p}, \mathbf{R}) \quad  \longmapsto \quad y = \begin{bmatrix}
V \\ W \\ \Gamma \\ \Upsilon
\end{bmatrix}.
\end{align}
The first six governing equations of the new system \eqref{eq:intrinsic_original} are derived from the governing equations of \eqref{eq:GEB}, while the last six originate from the definition of $\Gamma$ and $\Upsilon$, and are sometimes called \textit{compatibility conditions}. The relationship between the initial data of both systems is
\begin{equation} \label{eq:y0}
y^0 = \begin{bmatrix}
(\mathbf{R}^0)^{\intercal} v^0 \\
(\mathbf{R}^0 )^{\intercal} w^0 \\
(\mathbf{R}^0)^{\intercal} \frac{\mathrm{d}}{\mathrm{d}x} \mathbf{p}^0 - e_1 \\
\mathrm{vec}\left( (\mathbf{R}^0)^{\intercal} \frac{\mathrm{d}}{\mathrm{d}x}\mathbf{R}^0 - R^{\intercal}\frac{\mathrm{d}}{\mathrm{d}x} R \right)
\end{bmatrix}.
\end{equation}


\subsubsection*{The IGEB model}
The IGEB model may be obtained from the GEB model by the transformation described above or directly from mecahnics as in \cite{hodges1990}. Its unknown state $y \colon [0, \ell]\times[0, T] \rightarrow \mathbb{R}^d$, where $d=12$, takes the form
\begin{align*}
y = \begin{bmatrix}
v \\ s
\end{bmatrix},
\end{align*}
where $v, s \colon [0, \ell]\times [0, T] \rightarrow \mathbb{R}^6$ are \emph{velocities} and \emph{strains} of the beam, respectively. More precisely, $v$ consists of the linear and angular velocities, while $s$ consists of the translational and rotational strains (see also \cref{sec:unknowns_coef}). The dynamics of these unknowns are given by the system
\begin{align}  \label{eq:intrinsic_original}
\begin{cases}
\partial_t y + A \partial_x y + \bar{B}(x)y = \bar{g}(y)   &\text{in } (0,\ell)\times(0,T)\\
v(\ell, \cdot) = 0 &\text{for } t \in (0,T) \\
- \mathbf{C}^{-1} s(0, \cdot) = - \mu v(0, \cdot) &\text{for } t \in (0,T)\\
y(\cdot,0) = y^0(\cdot)  &\text{for } x \in (0,\ell),
\end{cases}
\end{align}
where $y^0 \in H^1(0, \ell; \mathbb{R}^d)$ is an initial datum. 
The boundary conditions correspond to those given for the GEB model \eqref{eq:GEB}.
The beam is clamped at the boundary $x = \ell$, as seen by the homogeneous Dirichlet boundary condition. At $x=0$, 
a "Neumann" feedback
control of the form $u(t) = -\mu v(0, t)$ is applied, where $\mu \in \mathbb{R}^{6 \times 6}$ is the constant diagonal matrix defined by
\begin{align*} 
\mu = \mathrm{diag}\left( \mu_1, \mu_1, \mu_1, \mu_2, \mu_2, \mu_2 \right) \qquad \text{with }\mu_1, \mu_2 >0.
\end{align*}
and where, $\mu_1, \mu_2>0$ correspond to the feedback parameters introduced in \eqref{eq:feedback_geb}.
The beam is characterized by the so-called \emph{mass} and \emph{flexibility} matrices $\mathbf{M}, \mathbf{C} \in \mathbb{R}^{6 \times 6}$, defined in \eqref{eq:def_M_C}, which are both positive definite diagonal matrices depending on the beam parameters previously introduced. The beam is also characterized by\footnote{We may assume that $\mathbf{E}$ is of higher regularity: $\mathbf{E} \in C^k([0, \ell]; \mathbb{R}^{d\times d})$, for $k>1$.} $\mathbf{E} \in C^1([0, \ell], \mathbb{R}^{6 \times 6})$ which depends on the form of the beam before deformation (i.e. on the initial strains) and is defined in \eqref{eq:def_E_bold}. 

Let us describe the coefficients and some of their properties.
Denote by $\mathds{O}_n$ the square zero matrix of size $n$.
The coefficients $A \in \mathbb{R}^{d \times d}$ and $\bar{B} \in C^1([0,\ell]; \mathbb{R}^{d\times d})$ are defined by
\begin{align*} 
A = \begin{bmatrix}
\mathds{O}_6 & -(\mathbf{M}\mathbf{C})^{-1} \\
-\mathds{I}_6 & \mathds{O}_6
\end{bmatrix}, \quad \bar{B} = 
\begin{bmatrix}
\mathds{O}_{6} & -\mathbf{M}^{-1}\mathbf{E}\mathbf{C}^{-1} \\
\mathbf{E}^\intercal & \mathds{O}_{6} 
\end{bmatrix}. 
\end{align*}
Both $A$ and $\bar{B}$ depend on the material and geometry of the beam, while $\bar{B}$ additionally depends on the initial strains. System \eqref{eq:intrinsic_original} is hyperbolic, semilinear and $y\equiv 0$ is a steady state.
Indeed, we will see that the matrix $A$ is hyperbolic, in the sense that all its eigenvalues are real and one may find $d$ associated independent eigenvectors.
It is important to note that $\bar{B}$ has a specific structure which will be used in the proof of the main results, and is not small (thus the perturbation is not negligible).
The nonlinearity $\bar{g} \in C^\infty(\mathbb{R}^d; \mathbb{R}^d)$, is defined by
\begin{align} \label{eq:def_gbar}
\bar{g}(y) = \bar{\mathcal{G}}(y)y,
\end{align}
where, denoting $y = (\mathbf{y}_1^\intercal, \mathbf{y}_2^\intercal, \mathbf{y}_3^\intercal, \mathbf{y}_4^\intercal)^\intercal$, with $\mathbf{y}_i \in \mathbb{R}^3$ for $1\leq i \leq 4$,
\begin{align*}
\bar{\mathcal{G}}(y)= - \mathrm{diag}(\mathbf{M}^{-1}, \mathds{I}_6)
\begin{bmatrix}
\rho a \widehat{\mathbf{y}}_2 & \mathds{O}_3 & \mathds{O}_3 & \widehat{S_1 \mathbf{y}_3}\\
\mathds{O}_3 & \rho \widehat{\mathbf{y}}_2 J & \widehat{S_1 \mathbf{y}_3} & \widehat{S_2 \mathbf{y}_4} \\
\mathds{O}_3 & \mathds{O}_3 & \widehat{\mathbf{y}}_2 & \widehat{\mathbf{y}}_1\\
\mathds{O}_3 & \mathds{O}_3 & \mathds{O}_3 & \widehat{\mathbf{y}}_2
\end{bmatrix},
\end{align*}
and $J, S_1, S_2 \in \mathbb{R}^{3 \times 3}$, defined in \eqref{eq:def_J_S1_S2}, are positive definite diagonal matrices depending on the beam parameters.
The nonlinearity is quadratic. More precisely, one may easily see from the definition of $\bar{g}$, that its components $\bar{g }_i \in C^\infty(\mathbb{R}^d)$ for $1\leq i \leq d$, can be written in the form $\bar{g}_i(y) = \big \langle y \,, \bar{G}^i y \big \rangle
$, where $\bar{G}^i \in \mathbb{R}^{d \times d}$ is a constant symmetric matrix whose diagonal contain zeros only (implying that $\bar{G}^i$ is indefinite).
Both $\bar{g}$ and its Jacobian matrix are zero when evaluated at the origin.

\subsection{Main results}

We will need to define compatibility conditions for System \eqref{eq:intrinsic_original}. As for the unknown, we write the initial datum $y^0$ as
\begin{align*}
y^0 = \begin{bmatrix}
v^0 \\ s^0
\end{bmatrix}, \qquad \text{with } v^0, s^0 \colon [0, \ell] \rightarrow \mathbb{R}^6.
\end{align*}

\begin{definition}
\label{def:comp_cond}
We say that the initial datum $y^0\in H^1(0, \ell; \mathbb{R}^d)$ fulfills the zero-order compatibility conditions if
\begin{align} \label{eq:compat_0} 
v^0(\ell) = 0 \quad \text{and} \quad \mathbf{C}^{-1} s^0(0) = \mu v^0(0),
\end{align}
We say that $y^0\in H^2(0, \ell; \mathbb{R}^d)$ fulfills the first-order compatibility conditions if it fulfills \eqref{eq:compat_0} and, $y^1 \in H^1(0, \ell; \mathbb{R}^d)$ defined by 
\begin{align*}
y^1 = -  A \frac{\mathrm{d}y^0}{\mathrm{d}x} - \bar{B} y^0 + \bar{g}(y^0) = \begin{bmatrix}
v^1 \\s^1
\end{bmatrix}
\end{align*}
also fulfills \eqref{eq:compat_0}, where $v^0, s^0$ are replaced by $v^1, s^1$ respectively.
\end{definition}

The local existence and uniqueness of $C^0([0, T]; H^1(0, \ell ;\mathbb{R}^d))$ solutions to general one-dimensional semilinear hyperbolic systems, and $C^0([0, T];H^2(0, \ell ;\mathbb{R}^d))$ solutions in the quasilinear case, 
have been addressed in \cite{bastin2017exponential} and \cite[Appendix B]{Coron_2016_stab}, respectively\footnote{The local and semi-global existence and uniqueness of $C^1([0, \ell ]\times[0,T]; \mathbb{R}^d)$ solutions to general one-dimensional quasilinear hyperbolic systems have been addressed in \cite[Lem. 2.3, Th. 2.1]{wang2006exact} (which is an extension of \cite[Lem. 2.3, Th. 2.5]{li2010controllability} to nonautonomous systems), and these results apply to \eqref{eq:intrinsic_original} if $y^0$ fulfills the first-order compatibility conditions. We recall that semi-global existence means that for any $T>0$, if the initial datum is sufficiently small then there exists a unique solution until time $T$ in $C^1([0, \ell ]\times[0,T]; \mathbb{R}^d)$, see the above references for more detail.}.
Both results apply to System \eqref{eq:intrinsic_original}, yielding \cref{prop:well-posedness} below. This relies on the fact that $A$ is hyperbolic, and on writing \eqref{eq:intrinsic_original} in Riemann invariants (also called \emph{characteristic} or \emph{diagonal} form of \eqref{eq:intrinsic_original}) to verify that the system fits in the framework of \cite{Coron_2016_stab, bastin2017exponential}.

\begin{proposition}[Well-posedness] \label{prop:well-posedness}
Let $k \in \{1, 2\}$, and assume that $\bar{B} \in C^k([0, \ell]; \mathbb{R}^{d \times d})$. Then, there exists $\delta_0>0$ such that the following holds. For any $y^0 \in H^k(0, \ell; \mathbb{R}^d)$ satisfying $\|y^0\|_{H^k(0, \ell; \mathbb{R}^d)} \leq \delta_0$ and the $(k-1)$-order compatibility conditions,
there exists a unique solution $y \in C^0([0, T), H^k(0, \ell; \mathbb{R}^d))$ to \eqref{eq:intrinsic_original} with $T\in (0, +\infty]$. Moreover, $T = + \infty$ if
\begin{align*}
\|y(\cdot, t)\|_{H^k(0, \ell; \mathbb{R}^d)} \leq \delta_0, \qquad \text{for all }  t \in [0, T).
\end{align*}
\end{proposition}

\begin{definition}[Local exponential stability]
\label{def:loc_exp_stab}
Let $k \in \{1, 2\}$. The steady state $y\equiv 0$ of \eqref{eq:intrinsic_original} is said to be locally $H^k$ exponentially stable if there exist $\varepsilon>0$, $\alpha>0$ and $\eta \geq 1$ such that the following holds. Let  $y^0 \in H^k(0, \ell; \mathbb{R}^d)$ fulfill $\|y^0\|_{H^k(0, \ell; \mathbb{R}^d)}\leq \varepsilon$, and the $(k-1)$-order compatibility conditions of \eqref{eq:intrinsic_original}. 
Then, there exists a unique global in time solution $y \in C^0([0, +\infty); H^k(0, \ell; \mathbb{R}^d))$ to \eqref{eq:intrinsic_original}. Moreover,
\begin{align*}
\|y(\cdot, t)\|_{H^k(0, \ell; \mathbb{R}^d)} \leq \eta  e^{- \alpha t } \|y^0\|_{H^k(0, \ell; \mathbb{R}^d)}, \qquad \text{for all } \, t \in [0, +\infty).
\end{align*}
\end{definition}

We may now state our main results.

\begin{theorem} \label{thm:stabilization}
Let $k \in \{1, 2\}$ and assume that $\bar{B} \in C^k([0, \ell]; \mathbb{R}^{d \times d})$. For any feedback parameters $\mu_1, \mu_2>0$, the steady state $y \equiv 0$ of \eqref{eq:intrinsic_original} is locally $H^k$ exponentially stable.
\end{theorem}

\subsubsection*{Idea of the proof}
System \eqref{eq:intrinsic_original} has for unknown state the \emph{physical} variable $y$. To study the stabilization problem, using that the matrix $A$ is hyperbolic, we first write this system in Riemann invariants. For this new system, the unknown state is the \emph{diagonal} variable $r$.
The proof of stability amounts to finding a so-called $H^k$ \emph{quadratic Lyapunov functional}, namely a functional of the form 
\begin{align} \label{eq:form_Lyap}
\mathcal{L}(t) = \sum_{j=0}^k \int_0^\ell \left\langle \partial_t^j r(x,t) \,, Q(x) \partial_t^j r(x,t) \right\rangle dx
\end{align}
for all $t \in [0, T)$, which, when $r$ is in some ball of $C^{k-1}([0,\ell ]\times[0,T];\mathbb{R}^d)$, is equivalent to the squared $H^k(0, \ell; \mathbb{R}^d)$ norm of $r(\cdot, t)$ and has an exponential decay with respect to time. In \eqref{eq:form_Lyap}, $r \in C^0([0, T); H^k(0, \ell; \mathbb{R}^d))$ is the solution to \eqref{eq:intrinsic_original} in diagonal form. General criteria on $Q \in C^1([0, \ell]; \mathbb{R}^{d\times d})$ for the existence of such Lyapunov functionals for one-dimensional first-order semilinear and quasilinear hyperbolic systems are given in \cite[Th. 10.2]{bastin2017exponential} and \cite[Th. 6.10]{Coron_2016_stab}.
They take the form of matrix inequalities involving the boundary conditions (hence, the feedback control) of \eqref{eq:intrinsic_charact} as well as the coefficients appearing in the equations.
Consequently, it will be sufficient to look for $Q=Q(x)$ fulfilling these criteria. Our choice of $Q$ is strongly linked to the expression of the energy of the beam (the sum of the kinetic and elastic energy) which, as we will see, may be written in the form
\begin{align*}
\mathcal{E}^\mathcal{D}(t) = \int_0^\ell \left\langle r(x,t) \,, Q^\mathcal{D} r(x, t) \right\rangle dx,
\end{align*}
for some constant matrix $Q^\mathcal{D} \in \mathbb{R}^{d \times d}$. Indeed, we will use $Q(x) = W(x) Q^\mathcal{D}$ in the proof, with a specific choice of weight matrix $W(x) \in \mathbb{R}^{d \times d}$.

\begin{remark} A few remarks are in order.
\label{rem:extensions}
\begin{enumerate}[label=\arabic*)]
\item \label{rem:reg_sol} If $y \in C^0([0, +\infty);H^k(0, \ell; \mathbb{R}^d))$ is solution to \eqref{eq:intrinsic_original} for some $k \in \{1, 2\}$, then $y$ belongs to $C^{0}([0, \ell]\times[0,+\infty); \mathbb{R}^d)$ in the case $k=1$; while $y \in C^{1}([0, \ell]\times[0,+\infty); \mathbb{R}^d)$ in the case $k=2$, see \cite[Cor. B.2]{Coron_2016_stab} for detail. Moreover, there exists $\bar{\eta}>0$ such that
\begin{align*}
\|y\|_{C^{k-1}([0, \ell] \times [0, +\infty); \mathbb{R}^d)} \leq \bar{\eta} e^{- \alpha t} \|y^0\|_{H^k(0, \ell; \mathbb{R}^d)}.
\end{align*}

\item The nonlinearity $\bar{g} \in C^\infty(\mathbb{R}^d; \mathbb{R}^d)$ is fixed and defined by \eqref{eq:def_gbar}, however we do not make use of its specific value to obtain \cref{prop:well-posedness} and \cref{thm:stabilization} (for $k \in \{1, 2\}$). Both would still hold if the map $\bar{\mathcal{G}}$ in \eqref{eq:def_gbar} is replaced by any other  $\bar{\mathcal{G}} \in C^k(\mathbb{R}^d; \mathbb{R}^{d \times d})$ such that $\bar{\mathcal{G}}(0)=0$.

\item It is interesting to note that if one substitutes the matrix $\mu$ for the precise value $\mathbf{M}^{\frac{1}{2}}\mathbf{C}^{-\frac{1}{2}}$, then the boundary condition at $x=0$ in the IGEB model in Riemann invariants \eqref{eq:intrinsic_charact} takes the form $r_+(0, t)=0$. This amounts to giving a so-called \emph{transparent} boundary condition at $x=0$. 

\item \label{rem:freely_vibrating} As explained in \cref{rem:freebeam_GEB}, the beam is freely vibrating. It would be of interest to study \eqref{eq:intrinsic_original} with forces, moments applied to the beam, as it may be subjected to gravity or aerodynamic forces for instance. Then, the term
\begin{align*}
\mathbf{M}^{-1}\begin{bmatrix}
\bar{\Phi} \\ \bar{\Psi}
\end{bmatrix}
\end{align*}
appears in the first six equations of \eqref{eq:intrinsic_original}, where $\bar{\Phi}, \bar{\Psi}$ are $\mathbb{R}^3$-valued functions representing the \emph{body-attached}\footnote{The relationship between $\bar{\phi}, \bar{\psi}$ and $\bar{\Phi}, \bar{\Psi}$ is $\bar{\phi} = \mathbf{R}\bar{\Phi}$ and $\bar{\psi} = \mathbf{R}\bar{\Psi}$ (see \cref{rem:body_attached_var}).} external forces and moments, respectively.
Consequently, the study of the steady states and the decay of the beam energy (see \cref{sec:choice_fb}), may not be straightforward anymore.

\item According to \cite[Sec. 6.2.2]{Coron_2016_stab}, one may also obtain $H^k$ stabilization for \eqref{eq:intrinsic_original} with $k>2$, if $\bar{B} \in C^k([0, \ell]; \mathbb{R}^{d \times d})$
and if $(k-1)$-order compatibility conditions (extending \cref{def:comp_cond}) are fulfilled. The Lyapunov functional \eqref{eq:form_Lyap} with time derivatives of the solution up to order $k>2$, would then be used.

\item \cref{thm:stabilization} still holds if one gives an homogeneous Neumann condition $s(\ell, t) =0$ at $x=\ell$ (free end) instead of considering clamped end.
For the IGEB model in diagonal form \eqref{eq:intrinsic_charact} this amounts to substituting the present boundary condition for $r_-(\ell, t) = r_+(\ell, t)$, leaving \cref{prop:conservation_energy} and \cref{prop:coron_exp_stab_H1} unchanged.
For the GEB model \eqref{eq:GEB} this amounts to replacing the Dirichlet conditions with $\Gamma(\ell, t)=0$ and $\Upsilon(\ell, t)=0$, and the beam energy $\mathcal{E}^\mathcal{P}$ remains nonincreasing (see \cref{prop:energy_P_nonincreasing}).
\end{enumerate}
\end{remark}

Relying on the above theorem, we obtain the following result on the GEB model.

\begin{theorem} \label{thm:fromIGEBtoGEB}
Let $\mu_1, \mu_2>0$. There exists $\varepsilon>0$, $C_1>0$ and $C_2>0$ such that the following holds. Assume that $v^0, w^0 \in C^2([0, \ell];\mathbb{R}^3)$, $R, \mathbf{R}^0 \in H^3(0, \ell; \mathrm{SO}(3))$, $h^\mathbf{R}\in \mathrm{SO}(3)$, $\mathbf{p}^0 \in H^3(0, \ell; \mathbb{R}^3)$ and $h^\mathbf{p} \in \mathbb{R}^3$, with $h^\mathbf{p} = \mathbf{p}^0(\ell)$ and $h^\mathbf{R} = \mathbf{R}^0(\ell)$. Assume that the function $y^0 \in H^2(0, \ell; \mathbb{R}^d)$ defined by \eqref{eq:y0} fulfills the first-order compatibility conditions of \eqref{eq:intrinsic_original}, as well as $\|y^0\|_{H^2(0, \ell; \mathbb{R}^d)}\leq \varepsilon$. 

Then, there exists a unique solution $(\mathbf{p}, \mathbf{R}) \in C^2( [0, \ell]\times[0, +\infty); \mathbb{R}^3 \times \mathrm{SO}(3))$ to \eqref{eq:GEB} with feedback \eqref{eq:feedback_geb}. Furthermore, for all $(x, t) \in [0, \ell]\times[0, +\infty)$,
\begin{align*}
|\partial_t \mathbf{p}(x,t)| + \|\partial_t \mathbf{R}(x,t)\| + |\Gamma(x,t)| +  |\Upsilon(x,t)| \leq C_1 e^{-C_2 t}.
\end{align*}
\end{theorem}
Under the assumptions of \cref{thm:fromIGEBtoGEB}, one knows by \cref{thm:stabilization} that for $\varepsilon>0$ small enough there exists a unique solution $y \in C^0([0, +\infty); H^2(0, \ell; \mathbb{R}^{d}))$ to \eqref{eq:intrinsic_original}. 
The solution $(\mathbf{p}, \mathbf{R})$ given by \cref{thm:fromIGEBtoGEB} is in fact related to $y$ by the transformation $\mathcal{N}$ defined in \eqref{eq:transfo}.

\subsection{Mechanical setting}
\label{sec:unknowns_coef}
We will now present the meaning of the unknown states of the IGEB and GEB models. To this end, we begin by describing the geometry of the underlying beam.
The beam is idealized as a \emph{centerline}, and a family of \emph{cross sections}, the centerline running along the geometric centers of the cross sections. Let $\{e_i\}_{i=1}^3$ $= \{(1, 0, 0)^\intercal$, $(0, 1, 0)^\intercal$, $(0, 0, 1)^\intercal\}$.

Before deformation, the position of the centerline $p \colon [0, \ell] \rightarrow \mathbb{R}^3$ and the orientation of the cross sections, are both known. The latter is given by the columns $\{b^i\}_{i=1}^3$ of a rotation matrix $R \colon [0, \ell] \rightarrow \mathrm{SO}(3)$. We assume that $b^1 = \frac{\mathrm{d}p}{\mathrm{d}x}$, implying that $p$ is parametrized by its arclength. 
At any time $t>0$, the position $\mathbf{p} \colon [0, \ell]\times [0, T] \rightarrow \mathbb{R}^3$ of the centerline and the orientation of the cross sections, given by the columns $\{\mathbf{b}^i\}_{i=1}^3$ of a rotation matrix $\mathbf{R} \colon [0, \ell]\times[0, T] \rightarrow \mathrm{SO}(3)$, are both unknown.
As shear deformation is allowed, $\mathbf{b}^1$ is not necessarily tangent to the centerline.
The sets $\{b^i(x)\}_{i=1}^3$ and $\{\mathbf{b}^i(x,t)\}_{i=1}^3$ are body-attached basis, with origin $p(x)$ and $\mathbf{p}(x,t)$ respectively.

Let the set $\Omega_{s} \subset \mathbb{R}^3$ be the beam when straight, untwisted, and such that the position of its centerline is given by the map $x \mapsto x e_1$ defined on $[0, \ell]$. We may write it as $\Omega_{s} = \bigcup_{x \in [0,\ell]}\mathfrak{a}(x)$ where $\mathfrak{a}(x)$ is the cross section intersecting the centerline at $x e_1$. Then, the beam before deformation and the beam at time $t>0$ take the form $\Omega_c = \{\bar{p}(X) \colon X \in \Omega_s\}$ and $\Omega_t = \{\bar{\mathbf{p}}(X, t) \colon X \in \Omega_s\}$
respectively, where for $X = (x,\xi_2,\xi_3)^\intercal \in \Omega_{s}$
\begin{align*}
\bar{p}(X) = p(x) + R(x)(\xi_2 e_2 + \xi_3 e_3), \quad \bar{\mathbf{p}}(X,t) = \mathbf{p}(x,t) + \mathbf{R}(x,t)(\xi_2 e_2 + \xi_3 e_3). 
\end{align*}
We call $\Omega_s$, $\Omega_c$ and $\Omega_t$  \emph{straight-reference} configuration, \emph{curved-reference} configuration and \emph{current} configuration of the beam, respectively. We refer to \cref{fig:config_beam} for visualization.
The vector $\xi_2 e_2 + \xi_3 e_3$ is the position of $X$ within $\mathfrak{a}(x)$.

\begin{figure}[h]
\centering
\includegraphics[scale=0.875]{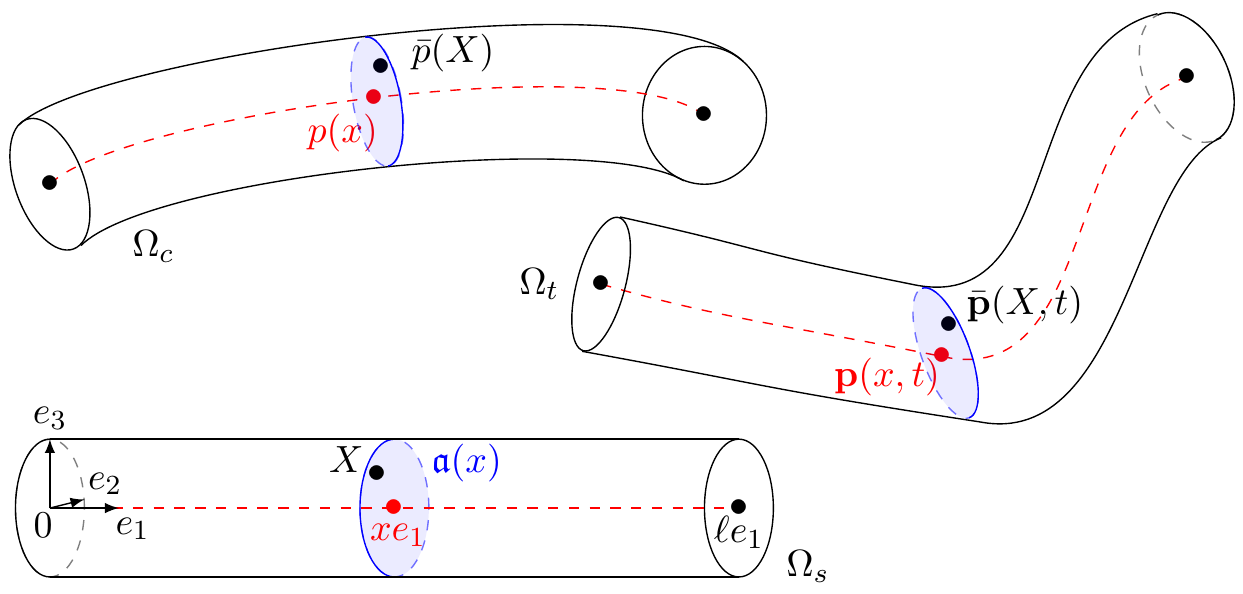}
\caption{The beam in straight-reference $\Omega_s$, curved-reference $\Omega_c$, and current $\Omega_t$ configurations.}
\label{fig:config_beam}
\end{figure}

\begin{remark}[Body-attached variable]
\label{rem:body_attached_var}
The unknown states of the IGEB model are \emph{body-attached variables} in the sense explained below.
We consider two kinds of coordinate systems: one is $\{e_i\}_{i=1}^3$ which is fixed in space and time, the other is the body-attached basis $\{\mathbf{b}^i\}_{i=1}^3$. 
We then make the difference between two kinds of vectors in $\mathbb{R}^3$: global and body-attached.
Consider two vectors $u := \sum_{i=1}^3 u_i e_i$ and $U:=\sum_{i=1}^3 U_i e_i$ of $\mathbb{R}^3$, 
the former being a \emph{global} vector and the latter being the \emph{body-attached} representation of $u$. By this, we mean that the components of $u$ are its coordinates with respect to the global basis $\{e_i\}_{i=1}^3$, while the components of $U$ are coordinates of the vector $u$ with respect to the body-attached basis $\{\mathbf{b}^i\}_{i=1}^3$. In other words $u = \sum_{i=1}^3 U_i \mathbf{b}^i $. Both vectors are then related by the identity $u = \mathbf{R} U$ since $\mathbf{b}^i = \mathbf{R} e_i$, and we may also call $u$ the \emph{global} representation of $U$.
\end{remark}

We have seen that the unknown state $y$ of the IGEB model \eqref{eq:intrinsic_original} consists of the velocities $v$ and strains $s$, and we want now to give the meaning of these variables in terms of positions and rotations. They are
\begin{align*}
v = \begin{bmatrix}
V \\ W
\end{bmatrix}, \qquad s = \begin{bmatrix}
\Gamma \\ \Upsilon
\end{bmatrix}
\end{align*}
where $V, W, \Gamma, \Upsilon \colon [0, \ell ]\times[0, T] \rightarrow \mathbb{R}^3$ are the functions defined in \eqref{eq:def_VWGU}, and are all body-attached variables.
One may note that the translational strain also writes as $\Gamma =\mathbf{R} ^\intercal \partial_x \mathbf{p} -  R  ^\intercal \tfrac{\mathrm{d}}{\mathrm{d}x} p$. 
The \textit{initial strain matrix} $\mathbf{E} \in C^1([0, \ell]; \mathbb{R}^{6 \times 6})$ is defined by
\begin{align} \label{eq:def_E_bold}
\mathbf{E} = \begin{bmatrix}
\widehat{\Upsilon}_c &  \mathds{O}_3\\
\widehat{e}_1 & \widehat{\Upsilon}_c
\end{bmatrix}, \qquad \text{where }\Upsilon_c = \mathrm{vec}\left(R^\intercal \tfrac{\mathrm{d}}{\mathrm{d}x} R \right).
\end{align}
The map $\Upsilon_c \colon [0, \ell] \rightarrow \mathbb{R}^3$ is the rotational strain in the curved-reference configuration (i.e. before deformation).
If the beam is straight and untwisted with centerline $p(x) = xe_1$ before deformation, then $R$ is the identity matrix and $\Upsilon_c = 0$.

Let us introduce the so-called mass matrix $\mathbf{M}$ and flexibility matrix $\mathbf{C}$. In general, for beams made of linear-elastic material, these matrices are positive definite (eventually, $\mathbf{M}$ positive semi-definite), symmetric and dependent on $x$. However, in this work, from our assumptions on the material and geometry of the beam, $\mathbf{M}, \mathbf{C} \in \mathbb{R}^{6 \times 6}$ are both positive definite constant diagonal matrices. They are defined by
\begin{align}\label{eq:def_M_C}
\mathbf{C} = \mathrm{diag}(S_1, S_2)^{-1}, \qquad \mathbf{M} = \rho \, \mathrm{diag}\big( a \mathds{I}_3, \ J \big),
\end{align} 
where $J \in \mathbb{R}^{3\times 3}$, called the \textit{inertia matrix}, and $S_1, S_2 \in \mathbb{R}^{3 \times 3}$, are positive definite diagonal matrices defined by
\begin{align} \label{eq:def_J_S1_S2}
J &= \mathrm{diag}\big((I_2+I_3)k_1, \ I_2, \ I_3 \big)\qquad \text{and} \qquad
\begin{aligned}
S_1 &=  a \, \mathrm{diag}(E, k_2 G, k_3 G) \\
S_2 &= J \, \mathrm{diag}(G, E, E).
\end{aligned}
\end{align}

\begin{remark}[Mass and flexibility matrices]\label{rem:mass_flex}
The flexibility matrix relates the stresses $F$ (i.e. vector of body-attached \emph{internal forces $\Phi$ and moments $\Psi$}) to the strains by $F = \mathbf{C}^{-1}s$, while the flexibility matrix relates the momenta $P$ to the velocities by $P = \mathbf{M}v$.
In this work, the unknown state $y$ consists of velocities and strains. One can choose the internal forces and moments $\Phi, \Psi$ as unknowns instead of strains $\Gamma, \Upsilon$ using the above relationship. The obtained system would have similar properties to \eqref{eq:intrinsic_original}.
\end{remark}

\subsection{Brief state of the art}

Up to the best of our knowledge, global in time existence and uniqueness of $C^0$ or $C^1$ solutions in $[0, \ell ]\times[0, \infty)$ to \eqref{eq:intrinsic_original} is not provided by general results present in the literature, even though one may find such results for quasilinear and semilinear problems similar to \eqref{eq:intrinsic_original}. For instance, in the case of initial value problems, \cite[Ch. 4]{Li_1994_global} assumes dissipativity of the lower order terms ($\bar{B}y$ and $\bar{g}(y)$ here) and \cite{beauchard2011large} gives a relaxation of this assumption, \cite{tartar1981some} considers $C^0(\mathbb{R}; L^1(\mathbb{R}; \mathbb{R}^n))$ solutions when there is not any \textit{linear} lower order term ($\bar{B}y$ here) and the quadratic term satisfies certain constraints (which are satisfied by $\bar{g}$ here); while in the case of initial boundary value problems, \cite[Ch. 5]{Li_1994_global} assumes dissipativity of the boundary conditions and the absence of \textit{linear} lower order terms, \cite{Kmit_classical_nonlin} gives a growth restriction on the lower order terms.

Stabilization of beam equations by means of feedback boundary controls goes back to \cite{quinnrussel78} for the string, \cite{kimrenardy87} for the Timoshenko beam; see also \cite{hegarty12, do18, morgul91, xu2005} and the references therein for other linear and nonlinear beam models. As metionned earlier, we focus of the Lyapunov approach to prove stability. For one-dimensional first-order hyperbolic systems, such as \eqref{eq:intrinsic_original}, several results of stabilization under boundary control are shown by means of quadratic Lyapunov functionals in \cite{Coron_2016_stab} and the references therein. There, when the system does not have any lower order term such as $\bar{B}y$ and $\bar{g}(y)$ here (systems of conservation laws), the exponential stability may rely on the dissipativity of the boundary conditions alone.
However, when lower order terms are present (systems of balance laws) the equations must also be taken into consideration. Some systems of nonlinear balance laws with a uniform steady state may be seen as systems of nonlinear conservation laws perturbed by the lower order terms: if the perturbation is small enough then the $C^1$-\,exponential stability is preserved, see \cite[Th. 6.1]{Coron_2016_stab}.
See also \cite{gugat2018} for two by two quasilinear systems with small lower order terms. 
System \eqref{eq:intrinsic_original} does have dissipative boundary conditions, however the perturbation is not small (see \cref{rem:perturbation}). Concerning general linear, semilinear and quasilinear systems, assumptions on both the boundary conditions and the system's coefficients are required in \cite[Pr. 5.1]{Coron_2016_stab},  \cite[Th. 10.2]{bastin2017exponential}, \cite{hayat2018exponential} and \cite[Th. 6.10]{Coron_2016_stab} for $L^2, H^1$, $C^1$ and $H^2$ exponential stability respectively.

\subsection{Notation}

Let $m, n \in \mathbb{N}$ and $M \in \mathbb{R}^{n\times n}$. Here, the identity and null matrices are denoted by $\mathds{I}_n \in \mathbb{R}^{n \times n}$ and $\mathds{O}_{n, m} \in \mathbb{R}^{n \times m}$, and we use the abbreviation $\mathds{O}_{n} = \mathds{O}_{n, n}$. The transpose, determinant and trace of $M$, and the matrix with components $|M_{i, j}|$ for $i, j \in \{1 \ldots n\}$, are denoted by $M^\intercal$, $\mathrm{det}(M)$, $\mathrm{tr}(M)$ and $|M|$ respectively. We use the notation $\|M\| = \sup_{|\xi| = 1} |M \xi|$, where $|\, . \,|$ is the Euclidean norm. The inner product in $\mathbb{R}^n$ is denoted $\langle \cdot \,, \cdot \rangle$.
The symbol $\mathrm{diag}(\, \cdot \, , \ldots, \, \cdot \, )$ denotes a (block-)diagonal matrix composed of the arguments.
We denote by $\mathcal{D}^+(n)$ the set of positive definite diagonal matrices of size $n$.
We denote by $\mathrm{Jac}_x f$ the Jacobian matrix of any $f = f(x)$ such that $f \in C^1(\mathbb{R}^n; \mathbb{R}^m)$.
Finally, we use the shortened notations $\mathbf{L}^2(0, \ell) = L^2(0, \ell; \mathbb{R}^d)$ and $\mathbf{H}^m(0,\ell)=H^m(0, \ell; \mathbb{R}^d)$.

\subsection{Outline}
In \cref{sec:energy_and_riem}, we explain how the feedback control is chosen (\cref{sec:choice_fb}), derive the diagonal form of System \eqref{eq:intrinsic_original} (\cref{sec:riem_transfo}), and study the energy for the diagonal system in order to gain information of use in the next section (\cref{sec:riem_energy}).
In \cref{sec:proof_main_th} and \cref{sec:fromIGEBtoGEB}, we prove the main results \cref{thm:stabilization} and \cref{thm:fromIGEBtoGEB}, respectively.

\section{Beam energy and Riemann invariants}
\label{sec:energy_and_riem}

\subsection{Choice of the feedback}
\label{sec:choice_fb}

To choose the boundary feedback control, we first look at the GEB model and the corresponding energy $\mathcal{E}^\mathcal{P}$ of the beam described below: we choose the feedback in such a way that the energy of the beam is nonincreasing. As the nonlinear transformation \eqref{eq:transfo} permits to obtain the IGEB model \eqref{eq:intrinsic_original} from the GEB model \eqref{eq:GEB}, we use the corresponding feedback for the IGEB model.
The energy is by definition
\begin{align} \label{eq:def_beam_energy}
\mathcal{E}^\mathcal{P}(t) =  \mathcal{K}(t) + \mathcal{V}(t).
\end{align}
where $\mathcal{K}$ is the \emph{kinetic energy} and $\mathcal{V}$ the \emph{elastic energy} (or \emph{strain energy}). For the type of beam considered here, $\mathcal{K}$ and $\mathcal{V}$ are given by
\begin{align} \label{eq:new_kin_ela_energy}
\begin{aligned}
\mathcal{K}(t) &= \int_0^\ell  \Big[\rho a |V(x, t)|^2 + \rho \big\langle W(x, t) \,, J W(x, t) \big\rangle \Big] dx\\
\mathcal{V}(t) &= \int_0^\ell  \Big[ \big\langle \Gamma(x, t) \,, S_1 \Gamma(x, t) \big \rangle + \big\langle \Upsilon(x, t) \,, S_2 \Upsilon(x, t) \big \rangle \Big] dx
\end{aligned}
\end{align}
where $V, W, \Gamma, \Upsilon$ are defined in \eqref{eq:def_VWGU}.

\begin{proposition} \label{prop:energy_P_nonincreasing}
If $\mathbf{p}\in C^2([0, \ell ]\times [0, T]; \mathbb{R}^3)$ and $\mathbf{R} \in C^2([0, \ell ]\times [0, T]; \mathrm{SO}(3))$ are solution to System \eqref{eq:GEB} with the boundary feedback control \eqref{eq:feedback_geb}, then $t \mapsto \mathcal{E}^\mathcal{P}(t)$ is nonincreasing on $[0, T]$.
\end{proposition}

\begin{proof}[Proof of \cref{prop:energy_P_nonincreasing}]

We study the derivative of the energy of the beam.
At first, we assume that external forces and moments are applied, as in \eqref{eq:geb_with_extforces}, in order to point out how considering a freely vibrating beam is of help. 
Let us denote $w = \mathbf{R}W$. After some substantial calculus making use, namely, of the definition of $V, W, \Gamma, \Upsilon$, of the invariance of the cross product in $\mathbb{R}^3$ under rotation (i.e. $\widehat{\mathbf{R}\xi} = \mathbf{R} \, \widehat{\xi} \, \mathbf{R}^\intercal$ for any $\xi \in \mathbb{R}^3$), of integration by parts,
and of the governing system \eqref{eq:geb_with_extforces}, we arrive at
\begin{align*}
\frac{\mathrm{d}}{\mathrm{d}t}\mathcal{E}^\mathcal{P}(t)
&= 2\int_0^\ell  \Big( \big\langle \partial_t \mathbf{p}(x,t) \,, \bar{\phi}(x,t) \big \rangle  + \big \langle w(x,t)\,, \bar{\psi}(x,t) \big \rangle \Big) dx \\
&\quad+ 2 \Big\langle \partial_t \mathbf{p}(\ell , t) \,, \mathbf{R}(\ell , t)S_1\Gamma(\ell , t)\Big \rangle + 2 \Big\langle \mathrm{vec}\big[ \mathbf{R}(\ell , t)^\intercal \partial_t \mathbf{R}(\ell , t) \big] \,, S_2 \Upsilon(\ell , t) \Big\rangle\\
&\quad - 2 \Big \langle\partial_t \mathbf{p}(0, t)\,, \mathbf{R}(0, t)S_1\Gamma(0, t) \Big \rangle - 2 \Big \langle w(0, t) \,, \mathbf{R}(0, t) S_2 \Upsilon(0, t) \Big \rangle.
\end{align*}
The boundary terms at $x=\ell$ are equal to zero since $\mathbf{p}(\ell, \cdot)$ and $\mathbf{R}(\ell, \cdot)$ are constant in time, while the integral is equal to zero since $\bar{\phi} \equiv \bar{\psi} \equiv 0$.
The last two terms in the above right-hand side are equal to $ 2 \left \langle\partial_t \mathbf{p}(0, \cdot)\,, h^1(\cdot) \right \rangle$ and $- 2 \left \langle w(0, \cdot) \,, h^2(\cdot) \right \rangle$, respectively. Hence, \eqref{eq:feedback_geb} yields that $\frac{\mathrm{d}}{\mathrm{d}t}\mathcal{E}^\mathcal{P}(t) = - 2 \left(\mu_1 |\partial_t \mathbf{p}(0, t)|^2 + \mu_2 |w(0, t|^2 \right)$ which is less than or equal to zero for all $t \in [0, T]$.
\end{proof}

\subsection{Transformation to Riemann invariants}
\label{sec:riem_transfo}

The following lemma, which follows from straightforward matrix multiplications, yields that $A$ is hyperbolic.

\begin{lemma}
Let $D \in \mathbb{R}^{6 \times 6}$ be the positive definite diagonal matrix defined by
\begin{align} \label{eq:def_D}
D =  (\mathbf{M} \mathbf{C})^{-1/2}.
\end{align}
Then, the matrix $A$ may be diagonalized as $A = L^{-1} \mathbf{D} L$, where the matrices $\mathbf{D}$, $L$, $L^{-1} \in \mathbb{R}^{d \times d}$ are defined by
\begin{align*} 
\mathbf{D} = \mathrm{diag}(-D, D), \qquad L = 
\begin{bmatrix}
\mathds{I}_6 & D \\
\mathds{I}_6 & -D
\end{bmatrix}, \qquad L^{-1} = \frac{1}{2} \begin{bmatrix}
\mathds{I}_6 & \mathds{I}_6\\
D^{-1} & - D^{-1}
\end{bmatrix}.
\end{align*}
\end{lemma}

Note that $D = \rho^{-\frac{1}{2}}\mathrm{diag}\left(E, \ k_2 G, \ k_3 G, \ G, \ E, \ E \right)^\frac{1}{2}$.
We will denote the diagonal entries of $\mathbf{D}$ by $\{\lambda_i\}_{i=1}^d$, as they are the eigenvalues of $A$. These include two repeated values since $\lambda_1 = \lambda_5= \lambda_6$ and $\lambda_i = -\lambda_{i-6}$ for $i > 6$. We have
\begin{align} \label{eq:def_lambdai}
\lambda_7 = \sqrt{\rho^{-1}E}, \quad \lambda_8 = \sqrt{\rho^{-1}k_2 G}, \quad \lambda_9 = \sqrt{\rho^{-1}k_3 G}, \quad \lambda_{10} = \sqrt{\rho^{-1}G}.
\end{align}

Applying the change of variable $r = L y$
in System \eqref{eq:intrinsic_original}, yields 
its diagonal form
\begin{align}
\label{eq:intrinsic_charact}
\begin{cases}
\partial_t r + \mathbf{D} \partial_x r + B(x)r = g(r) &\text{in } (0,\ell)\times(0,T)\\
r_-(\ell,t) = - r_+(\ell,t) &\text{for } t \in (0,T) \\
r_+(0,t) = \kappa \, r_-(0,t) &\text{for } t \in (0,T)\\
r(x,0) = r^0(x) &\text{for } x \in (0,\ell),
\end{cases}
\end{align}
with unknown state $r \colon [0, \ell]\times[0, T] \rightarrow \mathbb{R}^d$, where $r^0 = L y^0$, $B = L\bar{B}L^{-1}$ and $g(r) = L \bar{g}(L^{-1}r)$. In line with the sign of the diagonal entries of $\mathbf{D}$, we denote
\begin{align*}
r = \begin{bmatrix}
r_-\\r_+
\end{bmatrix}, \qquad \text{where }r_-, r_+ \in \mathbb{R}^6,
\end{align*}
for any $r \in \mathbb{R}^d$.
The map $B \in C^1([0,\ell ];\mathbb{R}^{d \times d})$ has the form
\begin{align*}
B = \begin{bmatrix}
 D \mathbf{E}^\intercal - \mathbf{M}^{-1} \mathbf{E} D \mathbf{M}  &  D \mathbf{E}^\intercal + \mathbf{M}^{-1}\mathbf{E} D \mathbf{M} \\
-D \mathbf{E}^\intercal - \mathbf{M}^{-1}\mathbf{E} D \mathbf{M} & -D \mathbf{E}^\intercal + \mathbf{M}^{-1}\mathbf{E} D \mathbf{M}
\end{bmatrix},
\end{align*}
for $\mathbf{E}$, $\mathbf{M}$ defined in \eqref{eq:def_E_bold}-\eqref{eq:def_M_C}. 
Note that $(B+B^\intercal)(x)$ is indefinite for all $x \in [0, \ell]$, since its trace is equal to zero. 
Similarly to $\bar{g}$, the nonlinear map $g \in C^\infty(\mathbb{R}^d; \mathbb{R}^d)$ has the form $g(r) = \mathcal{G}(r)r$ with $\mathcal{G}(r) = L \bar{\mathcal{G}}(L^{-1}r)$,
where $\bar{\mathcal{G}}$ is defined in \eqref{eq:def_gbar}.
Its components $g_i \in C^\infty(\mathbb{R}^d)$, for $1\leq i \leq d$, are quadratic forms with respect to $r \in \mathbb{R}^d$, as $g_i(r) = \big \langle r \,, G^i r \big \rangle$ where the constant symmetric matrix $G^i \in \mathbb{R}^{d \times d}$ is defined by
\begin{align*}
G^i = \begin{cases}
(L^{-1})^\intercal (\bar{G}^i + \lambda_{i+6} \bar{G}^{i+6}) L^{-1} & \text{if }i \leq 6\\
(L^{-1})^\intercal (\bar{G}^{i-6} - \lambda_{i} \bar{G}^{i}) L^{-1} & \text{if }i >6,
\end{cases}
\end{align*}
in terms of the symmetric matrices $\{\bar{G}^i\}_{i=1}^d$ which characterize $\bar{g}$.
Note that $r\equiv 0$ is a steady state of \eqref{eq:intrinsic_charact} and $(\mathrm{Jac}_r \, g)(0) = 0$.
The matrix $\kappa \in \mathbb{R}^{6 \times 6}$ is diagonal and depends on the feedback parameters $\mu_1, \mu_2>0$ introduced in \eqref{eq:feedback_geb}. It is defined by
\begin{align} \label{eq:def_kappa}
\kappa = (\mathbf{M} D + \mu )^{-1}(\mathbf{M} D - \mu ),
\end{align}
The diagonal entries of $\kappa$ belong to $(-1, 1)$ as they have the form $\frac{b-c}{b+c}$ for $b,c>0$.

\begin{remark} \label{rem:perturbation}
System \eqref{eq:intrinsic_original} has dissipative boundary conditions (see \cite[Sec. 4.1]{Coron_2016_stab}), in the sense that $\rho_{\infty}(K) := \inf \big\{\mathcal{R}_\infty(\Lambda K \Lambda^{-1}) \colon \Lambda \in \mathcal{D}^+(d) \big\} <1$, where
\begin{align*}
K = \begin{bmatrix}
\mathds{O}_6 & - \mathds{I}_6 \\
\kappa & \mathds{O}_6
\end{bmatrix}, \qquad \mathcal{R}_\infty(M) = \max_{1 \leq i \leq d}\sum_{j=1}^d|M_{ij}|.
\end{align*}
Indeed, for $\Lambda = \mathrm{diag}\left((1+\varepsilon)|\kappa|, \mathds{I}_6 \right)$, if $\varepsilon>0$ is small enough then $\mathcal{R}_\infty(\Lambda K \Lambda^{-1})<1$. However, the perturbation $(Br + g(r))$ is not small in general, in the sense that its derivative with respect to $r$ evaluated at zero (which is equal to $B$) may not be assumed arbitrarily small. Indeed, for instance, for a straight, untwisted beam with centerline $p(x) = xe_1$ before deformation, one has $\|B\| = \max \{\lambda_8, \ \lambda_9, \ a I_2^{-1} \lambda_9, \ a I_3^{-1} \lambda_8 \}$.
\end{remark}

\subsubsection*{Well-posedness and stabilization}

One may define compatibility conditions for System \eqref{eq:intrinsic_charact} similarly to \cref{def:comp_cond} for System \eqref{eq:intrinsic_original}.
For $k \in \{1, 2\}$, the initial datum $r^0 = L y^0$ fulfills the $(k-1)$-order compatibility conditions of \eqref{eq:intrinsic_original} if and only $r^0$ fulfills the $(k-1)$-order compatibility conditions of \eqref{eq:intrinsic_charact}.
As for \cref{prop:well-posedness}, \cite{Coron_2016_stab, bastin2017exponential} yields a local existence result for \eqref{eq:intrinsic_charact}. 
Furthermore, we can study the stability of the diagonal system \eqref{eq:intrinsic_original} in order to obtain the same result for the system in physical variables \eqref{eq:intrinsic_charact}. 

Indeed, let $k \in \{1, 2\}$. Assume that the steady state $r\equiv 0$ of \eqref{eq:intrinsic_charact} is locally $H^k$ exponentially stable, in the sense of \cref{def:loc_exp_stab} applied to \eqref{eq:intrinsic_charact} instead of \eqref{eq:intrinsic_original}.
In other words, assume that there exist $\varepsilon>0$, $\alpha>0$ and $\eta \geq 1$ such that for any $r^0 \in \mathbf{H}^k(0, \ell)$ fulfilling $\|r^0\|_{\mathbf{H}^k(0, \ell)} \leq \varepsilon$ and the $(k-1)$-order compatibility conditions of \eqref{eq:intrinsic_charact}, there exists a unique solution $r \in C^0([0, +\infty); \mathbf{H}^k(0, \ell))$ to \eqref{eq:intrinsic_charact}, and
\begin{align*}
\|r(\cdot, t)\|_{\mathbf{H}^k(0, \ell)} \leq \eta  e^{- \alpha t } \|r^0\|_{\mathbf{H}^k(0, \ell)}, \qquad \text{for all } \, t \in [0, +\infty).
\end{align*}
Let $\bar{\varepsilon}, \bar{\alpha}, \bar{\eta}>0$ be defined by $\bar{\varepsilon}  = \varepsilon \|L\|^{-1}$, $\bar{\alpha} = \alpha$ and $\bar{\eta} = \eta \|L \| \|L ^{-1}\|$.
Then, the steady state $y \equiv 0$ of the IGEB model \eqref{eq:intrinsic_original} is locally $H^k$ exponentially stable in the sense of \cref{def:loc_exp_stab} with the constants $(\varepsilon, \alpha, \eta)$ replaced with $(\bar{\varepsilon}, \bar{\alpha}, \bar{\eta})$.


\subsection{Energy of the beam}
\label{sec:riem_energy}

Here, we see that the boundary conditions of \eqref{eq:intrinsic_charact} are chosen in such a way that the energy of the beam is nonincreasing. From the definitions of $\mathbf{M},\mathbf{C}, S_1, S_2$ in \eqref{eq:def_M_C}-\eqref{eq:def_J_S1_S2}, one observes that the energy \eqref{eq:def_beam_energy}-\eqref{eq:new_kin_ela_energy} also writes as
\begin{align} 
\label{eq:energy_physical}
\mathcal{E}^\mathcal{P}(t) = \displaystyle \int_0^\ell 
\left\langle \begin{bmatrix}
V \\ W \\ \Gamma \\ \Upsilon
\end{bmatrix} \,, Q^\mathcal{P} \begin{bmatrix}
V \\ W \\ \Gamma \\ \Upsilon
\end{bmatrix} \right\rangle dx, \qquad \text{with }Q^\mathcal{P} = \mathrm{diag} \big( \mathbf{M}, \mathbf{C}^{-1} \big),
\end{align}
where $V, W, \Gamma, \Upsilon$ are defined in \eqref{eq:def_VWGU}.
Since the transformation from GEB to IGEB is $y = (V^\intercal, W^\intercal, \Gamma^\intercal, \Upsilon^\intercal)^\intercal$, and the change of variable $r = Ly$ leads to the diagonal form \eqref{eq:intrinsic_charact} of IGEB, we expect that the map $t \mapsto \mathcal{E}^\mathcal{D}(t)$ defined by
\begin{align} \label{eq:def_L0}
\mathcal{E}^\mathcal{D}(t) = \int_0^\ell  \big \langle r(x,t)\,, Q^\mathcal{D} r(x,t) \big \rangle dx, \qquad \text{with} \ Q^\mathcal{D} = (L^{-1})^\intercal Q^\mathcal{P} L^{-1},
\end{align}
is also nonincreasing if $r$ is solution to \eqref{eq:intrinsic_charact}, as the definitions of $\mathcal{E}^\mathcal{D}$ and $\mathcal{E}^\mathcal{P}$ coincide. As $\mathbf{C}^{-1} = D^2 \mathbf{M}$ (see \eqref{eq:def_D}), we observe that $Q^\mathcal{D}$ rewrites as:
\begin{align} \label{eq:def_QD}
Q^\mathcal{D} &= \frac{1}{4} \begin{bmatrix}
\mathds{I}_6 & D ^{-1}\\
\mathds{I}_6 & -D ^{-1}
\end{bmatrix}
\begin{bmatrix}
\mathbf{M} & \mathds{O}_6\\
\mathds{O}_6 & D^2 \mathbf{M}
\end{bmatrix}
\begin{bmatrix}
\mathds{I}_6 & \mathds{I}_6\\
D ^{-1} & - D ^{-1}
\end{bmatrix}
= \frac{1}{2} \begin{bmatrix}
\mathbf{M} & \mathds{O}_6 \\
\mathds{O}_6 & \mathbf{M}
\end{bmatrix}.
\end{align}

For the sake of clarity and in order to illustrate the structure of the coefficients, we provide a proof below. In particular, the following proposition implies that if $r \in C^1([0, \ell ]\times[0, T]; \mathbb{R}^d)$ is solution to \eqref{eq:intrinsic_charact} then $\|r(\cdot, t)\|_{\mathbf{L}^2(0, \ell )}$ is bounded on $[0, T]$.

\begin{proposition} \label{prop:conservation_energy}
Assume that $r$ is the unique solution to \eqref{eq:intrinsic_charact} in $C^1([0,\ell ]\times[0,T]; \mathbb{R}^d)$. Then, the map $t \mapsto \mathcal{E}^\mathcal{D}(t)$, defined by \eqref{eq:def_L0}, is nonincreasing on $[0,T]$.
\end{proposition}

\begin{proof}
Let $r$ be as in \cref{prop:conservation_energy}.
Using the governing system, integration by parts with the fact that $Q^\mathcal{D}$ and $\mathbf{D}$ commute, one deduces that
\begin{align*}
\frac{\mathrm{d}\mathcal{E}^\mathcal{D}}{\mathrm{d}t} 
&= \frac{1}{2} \Big[ \big \langle r_-(\ell,t)\,, \mathbf{M} D r_-(\ell,t) \big \rangle -  \big \langle r_+(\ell,t)\,, \mathbf{M} D r_+(\ell,t) \big \rangle \\
&\hspace{-0.5cm} - \big \langle r_-(0,t)\,, \mathbf{M} D r_-(0,t) \big \rangle +  \big \langle r_+(0,t)\,, \mathbf{M} D r_+(0,t)\big \rangle \Big]  + 2\int_0^\ell  \big\langle r\,, Q^\mathcal{D} (Br + g(r)) \big \rangle dx.
\end{align*}
Using the boundary conditions, the boundary terms in the above right-hand side reduce to $\frac{1}{2} \left \langle r_-(0,t) \,,  (\kappa^2  - \mathds{I}_6) \mathbf{M} D r_-(0,t) \right \rangle$, which is nonpositive since the diagonal entries of $\kappa$ belong to $(-1, 1)$. It remains to see that the last term of above right-hand side is null. The product $Q^\mathcal{D}B$ is skew-symmetric since it writes as
\begin{align}\label{eq:prod_QD_B}
Q^\mathcal{D} B = 
\begin{bmatrix}
-(\mathcal{B} - \mathcal{B}^\intercal) & -(\mathcal{B} + \mathcal{B}^\intercal)\\
\mathcal{B} + \mathcal{B}^\intercal & \mathcal{B} - \mathcal{B}^\intercal
\end{bmatrix}, \quad \text{for} \quad \mathcal{B} = \tfrac{1}{2}\mathbf{E}D\mathbf{M},
\end{align}
hence $ \langle r \,, Q^\mathcal{D}B r \rangle = 0$ for all $r \in \mathbb{R}^d$. By definition of $Q^\mathcal{D}$ and $g$, $\langle r\,, Q^\mathcal{D} g(r) \rangle = 0$ for all $r \in \mathbb{R}^d$ if and only if $\langle y \,, Q^\mathcal{P} \bar{g}(y)\rangle = 0$ for any $y \in \mathbb{R}^d$. The latter holds directly by definition of $Q^\mathcal{P}$ and $\bar{g}$.
Indeed, denoting $y = (\mathbf{y}_1^\intercal, \mathbf{y}_2^\intercal, \mathbf{y}_3^\intercal, \mathbf{y}_4^\intercal)^\intercal$, with $\mathbf{y}_1, \mathbf{y}_2, \mathbf{y}_3, \mathbf{y}_4 \in \mathbb{R}^3$, one has
\begin{align*}
\big \langle y\,, Q^\mathcal{P} \bar{g}(y)\big \rangle 
=& - \rho a \big \langle  \mathbf{y}_1  \,,  \widehat{\mathbf{y}}_2\mathbf{y}_1  \big \rangle -
\rho \big \langle   \mathbf{y}_2 \,, \widehat{\mathbf{y}}_2 J \mathbf{y}_2 \big \rangle -
\big \langle  \mathbf{y}_1  \,,  \widehat{S_1 \mathbf{y}_3} \mathbf{y}_4  \big \rangle -
\big \langle  \mathbf{y}_2  \,,  \widehat{S_1\mathbf{y}_3}\mathbf{y}_3  \big \rangle \\
& - \big \langle  \mathbf{y}_2  \,,  \widehat{S_2 \mathbf{y}_4} \mathbf{y}_4  \big \rangle -
\big \langle  \mathbf{y}_3  \,,  S_1 \widehat{\mathbf{y}}_2 \mathbf{y}_3  \big \rangle -
\big \langle  \mathbf{y}_3  \,,  S_1 \widehat{\mathbf{y}}_1 \mathbf{y}_4  \big \rangle -
\big \langle \mathbf{y}_4   \,,  S_2\widehat{\mathbf{y}}_2 \mathbf{y}_4  \big \rangle,
\end{align*}
where the first two terms of the above right-hand side are null, while the remaining terms also writes as the sum of $\big \langle \widehat{S_1\mathbf{y}_3} \mathbf{y}_1 + \widehat{S_2 \mathbf{y}_4} \mathbf{y}_2 + \widehat{\mathbf{y}}_1 S_1 \mathbf{y}_3 + \widehat{\mathbf{y}}_2 S_2 \mathbf{y}_4 \,,  \mathbf{y}_4  \big \rangle $ and $\big \langle  \widehat{S_1 \mathbf{y}_3}\mathbf{y}_2 + \widehat{y}_2 S_1 \mathbf{y}_3 \,,  \mathbf{y}_3  \big \rangle$ which are both equal to zero.
%
\end{proof}

\begin{remark}[Structure of $B$] \label{rem:special_form_B}
An interest in going through the proof of \cref{prop:conservation_energy} is a resulting observation on the structure of $B$ that will be used in the proof of \cref{thm:stabilization}.
While $B$ is neither skew-symmetric nor positive or negative semi-definite, one observes that the product $Q^\mathcal{D} B$ not only is skew-symmetric, but also has the specific form \eqref{eq:prod_QD_B}.
\end{remark}


\section{Proof of \cref{thm:stabilization}}
\label{sec:proof_main_th}

We will show that the steady state $r\equiv 0$ of \eqref{eq:intrinsic_charact} is locally $H^1$ and $H^2$ exponentially stable, in order to prove \cref{thm:stabilization}.

\subsection{Strategy and proof}
\label{sec:strategy_proof}

Applying \cref{prop:coron_exp_stab_H1} given below is sufficient to prove the main result \cref{thm:stabilization}, and is equivalent to finding a quadratic Lyapunov functional for System \eqref{eq:intrinsic_charact}.

For any $M \in \mathcal{D}^+(d)$, we denote $M = \mathrm{diag}(M_-, M_+)$, where $M_-, M_+ \in \mathcal{D}^+(6)$. 

\begin{proposition} \label{prop:coron_exp_stab_H1}
Assume that $B \in C^k([0, \ell]; \mathbb{R}^{d \times d})$. Assume that there exists $Q \in C^1([0,\ell ]; \mathcal{D}^+(d))$ such that:
\begin{align} \label{eq:coron_th_bound_cond}
\kappa^2 Q_+(0) - Q_-(0) \quad \text{and} \quad Q_-(\ell ) - Q_+(\ell ) \quad \text{are negative semi-definite};
\end{align}
and, for any $x \in [0, \ell ]$,
\begin{align} 
\label{eq:coron_th_struct_cond}
\tfrac{\mathrm{d}}{\mathrm{d}x}Q(x) \mathbf{D} - Q(x)B(x) - B(x)^\intercal Q(x) \quad \text{is negative definite.}
\end{align}
Then, the steady state $r\equiv0$ of \eqref{eq:intrinsic_charact} is locally $H^1$ and $H^2$ exponentially stable.
\end{proposition}

\begin{remark}
The condition \eqref{eq:coron_th_bound_cond} concerns the feedback control, while \eqref{eq:coron_th_struct_cond} concerns coefficients appearing in the equations.
The conditions for $H^k$ stability are the same for both order $k \in \{1, 2\}$ and we also use the same feedback control.
\end{remark}

\begin{proof}[Proof of \cref{prop:coron_exp_stab_H1}]
This proposition is a special case of the general results \cite[Th. 10.2]{bastin2017exponential} for $k=1$ and \cite[Th. 6.10]{Coron_2016_stab} for $k=2$, given for one-dimensional first-order semilinear and quasilinear hyperbolic systems, respectively. In these results, exponential stability is granted if the matrix $\tfrac{\mathrm{d}}{\mathrm{d}x}Q(x) \mathbf{D} - Q(x)F(x) - F(x)^\intercal Q(x)$,
where $F(x)$ is the derivative with respect to $r$ of the perturbation (i.e. the lower order terms) evaluated at zero, 
and the matrix
\begin{align*} 
- \begin{bmatrix}
Q_-(0) D  & \mathds{O}_6 \\
\mathds{O}_6 & Q_+(\ell ) D 
\end{bmatrix}
 + \begin{bmatrix}
\mathds{O}_6 & - \mathds{I}_6 \\
\kappa & \mathds{O}_6
\end{bmatrix}^\intercal 
\begin{bmatrix}
Q_-(\ell ) D  & \mathds{O}_6 \\
\mathds{O}_6 & Q_+(0) D 
\end{bmatrix}
\begin{bmatrix}
\mathds{O}_6 & - \mathds{I}_6 \\
\kappa & \mathds{O}_6
\end{bmatrix}
\end{align*}
are negative definite and negative semi-definite, respectively.
Since $(\mathrm{Jac}_rg )(0) = 0$, the derivative of $(Br - g (r))$ with respect to $r$, evaluated at $r\equiv 0$, is equal to $B$, yielding the condition \eqref{eq:coron_th_struct_cond}. 
Moreover, the second matrix also writes as the product $\mathrm{diag}\left( \kappa^2Q_+(0) - Q_-(0) \,, \ Q_-(\ell ) - Q_+(\ell ) \right) \mathrm{diag}\left( D\,, D \right)$
which is negative semi-definite if and only if \eqref{eq:coron_th_bound_cond} holds.
\end{proof}

Our objective is to apply \cref{prop:coron_exp_stab_H1} in order to prove \cref{thm:stabilization}.\\


Let us recall the notation we have used so far. The mass matrix $\mathbf{M}$ and well as the inertia matrix $J$ are defined in \eqref{eq:def_J_S1_S2}. Their diagonal entries are denoted $\{\mathbf{M}_i\}_{i=1}^3$ and $\{J_i\}_{i=1}^3$. The initial strain matrix $\mathbf{E} = \mathbf{E}(x)$ and initial curvature $\Upsilon_c = \Upsilon_c(x)$ are defined in \eqref{eq:def_E_bold}, and the components of the latter are denoted $\{\Upsilon_{ci}\}_{i=1}^3$. The matrix $D$ and the eigenvalues $\{\lambda_i\}_{i=1}^d$ of $A$ are defined in \eqref{eq:def_D} and \eqref{eq:def_lambdai}.

We now introduce some matrices, constants and functions that will be involved in the proof of \cref{thm:stabilization}.
The constant $C_\kappa \in (0, 1)$, which characterizes the matrix $\kappa$ defined in \eqref{eq:def_kappa}, is defined by
\begin{align} \label{eq:def_C_kappa}
C_\kappa = \max_{1 \leq i \leq 6}\kappa_i^2,
\end{align}
where $\{\kappa_i\}_{i=1}^6$ are the diagonal entries of $\kappa$. 
The positive definite diagonal matrix $\Lambda \in \mathbb{R}^{d \times d}$, depends on the beam parameters and is defined by 
\begin{align} \label{eq:def_Lam}
\Lambda = \mathrm{diag}(\mathbf{M} D , \mathbf{M} D ).
\end{align}
The map $\Theta \in C^1([0, \ell ]; \mathbb{R}^{d \times d})$, depends not only on the beam parameters but also on the initial strains. It is defined by
\begin{align} \label{eq:def_The}
\Theta = - \begin{bmatrix}
\mathds{O}_6 & \mathbf{E}D\mathbf{M} + (\mathbf{E}D\mathbf{M})^\intercal \\
\mathbf{E}D\mathbf{M} + (\mathbf{E}D\mathbf{M})^\intercal & \mathds{O}_6
\end{bmatrix}.
\end{align}
Observe that $\Theta$ is indefinite, as it is symmetric and its trace equals zero.
Moreover, for any $x \in [0, \ell]$, the largest eigenvalue of $\Theta(x)$ is denoted $\sigma^{\Theta(x)}_d$. The map $x \mapsto \sigma^{\Theta(x)}_d$ is continuous on $[0, \ell ]$ since $x \mapsto \Theta(x)$ belongs to $C^0([0, \ell ];\mathbb{R}^{d \times d})$, see \cite[Coro. VI.1.6]{bhatia_matrix}.

Let us now introduce the continuous functions $q_1, q_2 \in C^0([0, \ell])$ which are the object of the next lemma. They are defined by
\begin{align}\label{eq:def_q1_q2}
q_1(x) = \max_{1 \leq i \leq  6} \theta_i(x) , \qquad q_2(x) = \sigma^{\Theta(x)}_d  \Big(\min_{1 \leq i \leq 6}\mathbf{M}_i \lambda_{i+6}\Big)^{-1},
\end{align}
where $\{\theta_i\}_{i=1}^6 \subset C^0([0, \ell ])$ are the following nonnegative functions:
\begin{align*}
\Scale[0.89]{
\begin{aligned}
\theta_1 &= \left|1 - \lambda_8 \lambda_7^{-1} \right| |\Upsilon_{c3}|+ \left|1 - \lambda_9 \lambda_7^{-1} \right|| \Upsilon_{c2}|, \qquad \ 
\theta_4 = \Big| 1 -  \frac{\lambda_7 J_2}{\lambda_{10} J_1} \Big|| \Upsilon_{c3}| + \Big|1- \frac{\lambda_7 J_3}{\lambda_{10} J_1} \Big| |\Upsilon_{c2}|, \\
\theta_2 &=  \left|1 -  \lambda_7  \lambda_8^{-1} \right|| \Upsilon_{c3}| + \left|1 - \lambda_9 \lambda_8^{-1}\right|| \Upsilon_{c1}| + 1,\ \
\theta_5 = \frac{ a \lambda_9}{\lambda_7 J_2} + \Big|1- \frac{\lambda_{10} J_1}{\lambda_7 J_2}\Big|| \Upsilon_{c3}| + \Big|1 -  \frac{J_3}{J_2} \Big||\Upsilon_{c1}|,\\
\theta_3 &= \left|1 - \lambda_7 \lambda_9^{-1} \right| |\Upsilon_{c2}| +  \left|1 - \lambda_8 \lambda_9^{-1} \right|| \Upsilon_{c1}| + 1, \ \  
\theta_6 =  \frac{a \lambda_8}{\lambda_7 J_3}  + \Big|1 -  \frac{\lambda_{10} J_1}{\lambda_7 J_3} \Big|| \Upsilon_{c2}| + \Big|1- \frac{J_2}{J_3} \Big|| \Upsilon_{c1}|.
\end{aligned}}
\end{align*}

\begin{lemma} \label{lem:neg_Psi}
Let $m \in \{1, 2\}$. Assume that $w_-, w_+ \in C^1([0,\ell ])$ are positive functions such that $w_+(\ell ) \leq w_-(\ell )$ and
\begin{align*}
\frac{\mathrm{d}w_-}{\mathrm{d}x}>0, \quad \frac{\mathrm{d}w_+}{\mathrm{d}x} <0, \quad \min\left\{ \left|\tfrac{\mathrm{d}w_-}{\mathrm{d}x} \right|, \left|\tfrac{\mathrm{d}w_+}{\mathrm{d}x} \right| \right\} > (w_+ -w_-) q_m,  \qquad \text{in } [0,\ell ].
\end{align*}
Then, for all $x \in [0,\ell ]$, the matrix $\mathcal{S} := \mathrm{diag}\big(- \tfrac{\mathrm{d}w_-}{\mathrm{d}x}\mathds{I}_6, \tfrac{\mathrm{d}w_+}{\mathrm{d}x}  \mathds{I}_6\big) \Lambda + (w_+ - w_-)\Theta$
is negative definite. 
\end{lemma}

\begin{proof}[Proof of \cref{lem:neg_Psi}]
Let $x \in [0,\ell ]$. Let us start with $m = 1$. By \cite[Def. 6.1.9, Coro. 7.2.3]{horn2012matrix}, if a matrix is strictly diagonally dominant with negative diagonal entries, then it is negative definite. Since $\mathrm{diag}\big(-\tfrac{\mathrm{d}w_-}{\mathrm{d}x}(x) \mathds{I}_6, \tfrac{\mathrm{d}w_+}{\mathrm{d}x}(x) \mathds{I}_6 \big) \Lambda$ is negative definite and the diagonal entries of $(w_+(x) - w_-(x))\Theta(x)$ are null, $\mathcal{S}(x)$ is negative definite if
\begin{align} \label{eq:ineq_diag_dom}
\left(\mathrm{diag}\left(\left|\tfrac{\mathrm{d}w_-}{\mathrm{d}x}(x)\right| \mathds{I}_6, \left|\tfrac{\mathrm{d}w_+}{\mathrm{d}x}(x)\right| \mathds{I}_6 \right) \Lambda \right)_i> (w_+(x)-w_-(x)) \sum_{j=1}^d|\Theta_{ij}(x)|
\end{align}
holds for all $i \in \{ 1 \ldots d \}$. 
By the definition of $\Theta$ and $\Lambda$, \eqref{eq:ineq_diag_dom} is equivalent to 
\begin{align*}
\min \left\{ \left|\tfrac{\mathrm{d}w_-}{\mathrm{d}x}\right|, \left|\tfrac{\mathrm{d}w_+}{\mathrm{d}x}\right| \right\} > (w_+ -w_-) \theta_i, \quad \ \text{where } \ \theta_i = \sum_{j=1}^6 \Big| \big(\mathbf{M}^{-1} D \mathbf{E} D \mathbf{M} + \mathbf{E}^\intercal  \big)_{i,j} \Big|,
\end{align*}
holding for all $i \in \{ 1 \ldots d \}$, where we omitted the argument $x$ for clarity. It remains to look into the definition of $\mathbf{E},\mathbf{M}$ and $D$ to deduce that each $\theta_i$ of the above equation has the form given above \cref{lem:neg_Psi}.
This finishes the proof for the case $m = 1$.

We now consider the case $m=2$. Denote by $\{\sigma^M_i\}_{i=1}^n$ the eigenvalues of an Hermitian matrix $M \in \mathbb{R}^{n \times n}$ in nondecreasing order (then the largest eigenvalue of $M$ is $\sigma^M_n$). Weyl's Theorem  \cite[Th. 4.3.1, Coro. 4.3.15]{horn2012matrix} provides the bound $\sigma^{M_1+M_2}_i \leq \sigma^{M_1}_i + \sigma^{M_2}_n$ on the eigenvalues of the sum of Hermitian matrices $M_1, M_2 \in \mathbb{R}^{n \times n}$.
Hence, the eigenvalues of $\mathcal{S}(x)$ necessarily satisfy
\begin{equation*}
\sigma^{\mathcal{S}(x)}_i \leq - \min \left\{ \left|\tfrac{\mathrm{d}w_-}{\mathrm{d}x}(x)\right|, \left|\tfrac{\mathrm{d}w_+}{\mathrm{d}x}(x)\right| \right\} \Big( \min_{1 \leq i \leq 6}\mathbf{M}_i \lambda_{i+6} \Big) + (w_+(x) - w_-(x)) \sigma^{\Theta(x)}_d,
\end{equation*}
and they are all negative if the above right-hand side is negative.
\end{proof}

\begin{lemma} \label{lem:existence_varphi}
Let $c>0$. There exists $\varphi  \in C^1([0, \ell ])$ such that
\begin{align} \label{eq:g_properties}
\varphi (x)>0, \quad \frac{\mathrm{d}\varphi}{\mathrm{d}x}(x)>0, \quad \frac{\mathrm{d}\varphi}{\mathrm{d}x}(x) > 2 c \, (\varphi (\ell ) - \varphi (x)), \qquad \text{for all }x \in [0,\ell ],
\end{align}
and $0 < \varphi (0) < \varphi (\ell )$ may be chosen arbitrarily.
\end{lemma}

\begin{proof}[Proof of \cref{lem:existence_varphi}]
Notice that \eqref{eq:g_properties} is equivalent to
\begin{align} \label{eq:g_properties_equiv}
0 < \varphi (0) < \varphi (\ell ), \qquad \tfrac{\mathrm{d}}{\mathrm{d}x} \varphi(x) > 2 c \, (\varphi (\ell ) - \varphi (x)), \qquad \text{for all }x \in [0, \ell ].
\end{align}
Let $\alpha = 2c$. The inequality $\tfrac{\mathrm{d}\varphi}{\mathrm{d}x}  > 2c (\varphi (\ell ) - \varphi )$ is equivalent to
\begin{align} \label{eq:cond_gprime_ineq1}
e^{\alpha x}(\tfrac{\mathrm{d}}{\mathrm{d}x}\varphi (x) + \alpha (\varphi (x) - \varphi (\ell )))>0, \qquad \text{for all }x \in [0, \ell ].
\end{align}
In the above left-hand side, one recognizes  the derivative of
$x \mapsto e^{\alpha x}(\varphi (x) - \varphi (\ell ))$, hence \eqref{eq:cond_gprime_ineq1} holds if and only if $\frac{\mathrm{d}}{\mathrm{d}x} \left( e^{\alpha x}(\varphi (x) - \varphi (\ell )) \right) \geq \varepsilon$ for some $\varepsilon>0$.
Integrating this inequality over $[0, x]$ and isolating the term $\varphi (x)$ on one side, this is equivalent to $\varphi (x) \geq \varphi (\ell ) - e^{-\alpha x} (\varphi (\ell ) - \varphi (0) - \varepsilon x)$.
We choose as a candidate the function $\varphi (x) := \varphi (\ell ) - e^{-\alpha x} (\varphi (\ell ) - \varphi (0) - \varepsilon x)$.
Equality at $x=\ell $ is true if and only if $\varepsilon = \frac{\varphi (\ell ) - \varphi (0)}{\ell }$, which is positive if and only if $\varphi (\ell )>\varphi (0)$. Then, $\varphi $ writes as
\begin{align} \label{eq:existence_g_proof}
\varphi (x) = \varphi (\ell ) - e^{-\alpha x}\left( 1 - x \ell ^{-1}\right)(\varphi (\ell ) - \varphi (0)),
\end{align}
and fulfills $\tfrac{\mathrm{d}\varphi}{\mathrm{d}x}>2c(\varphi (\ell ) - \varphi )$ by construction. Assuming that $\varphi (0)>0$, the function $\varphi $ defined by \eqref{eq:existence_g_proof} now satisfies \eqref{eq:g_properties_equiv}, and this concludes the proof.
\end{proof}


We are in position to prove the first main result.

\begin{proof}[Proof of \cref{thm:stabilization}]
To apply \cref{prop:coron_exp_stab_H1}, one should find a map $Q \in C^1([0,\ell ]; \mathcal{D}^+(d))$ fulfilling the three matrix inequalities in \eqref{eq:coron_th_bound_cond}-\eqref{eq:coron_th_struct_cond}. Note that \eqref{eq:coron_th_struct_cond} cannot hold if $Q$ is constant, since the trace of $QB+B^\intercal Q$ is null (implying that this matrix is indefinite). Hence, it appears that $Q$ should be chosen in such a way that \eqref{eq:coron_th_struct_cond} holds due to the presence of $\tfrac{\mathrm{d}}{\mathrm{d}x}Q \mathbf{D}$.
We will proceed as follows:
\begin{enumerate}[label=\arabic*)]
\item Step 1: Based on \cref{rem:special_form_B}, we choose $Q = \mathrm{diag}(w_- \mathds{I}_6, w_+ \mathds{I}_6)Q^\mathcal{D}$ as an Ansatz, where $Q^\mathcal{D}$ is the matrix that characterizes the beam energy for the diagonal system \eqref{eq:intrinsic_charact}, and $w_-, w_+ \in C^1([0, \ell])$ are positive \emph{weights}. 
\item Step 2: As $QB+B^\intercal Q$ is indefinite, we choose the monotonicity of the weights in such a way that $\tfrac{\mathrm{d}}{\mathrm{d}x}Q \mathbf{D}$ is negative definite. By means of \cref{lem:neg_Psi}, we obtain more explicit conditions on the weights which are sufficient for the matrix inequalities \eqref{eq:coron_th_bound_cond}-\eqref{eq:coron_th_struct_cond} to be fulfilled.
\item Step 3: We show that such weights exist with the help of \cref{lem:existence_varphi}.
\end{enumerate}

\subsubsection*{Step 1: Ansatz for $Q$}

In \cref{sec:riem_energy}, we have seen that the energy of the beam for the IGEB model in Riemann invariants is characterized by the matrix $Q^\mathcal{D}$, defined in \eqref{eq:def_QD}. Furthermore, we have seen that the product $Q^\mathcal{D}B$ has the specific form given in \cref{rem:special_form_B}.
Let the functions $w_-, w_+ \in C^1([0, \ell ])$ be such that
\begin{align}\label{eq:weights_positive}
w_->0, \quad w_+>0, \quad \text{in }[0, \ell].
\end{align}
To simplify the task of finding $Q$, we choose the following ansatz for $Q$:
\begin{align} \label{eq:ansatz_Q}
Q(x) =  W(x) Q^\mathcal{D}, \qquad \text{where } \ W = \mathrm{diag} \left(w_-\mathds{I}_6, w_+ \mathds{I}_6 \right),
\end{align}
and $w_-, w_+$ are called \textit{weights}. The matrix $Q$ defined by \eqref{eq:ansatz_Q} fulfills the conditions of \cref{prop:coron_exp_stab_H1} if and only if
\begin{align} \label{eq:weights_cond_0_L}
w_+(0) \leq C_\kappa^{-1}w_-(0), \qquad w_-(\ell ) \leq w_+(\ell ),
\end{align} 
and, for any $x \in [0, \ell ]$, the matrix
\begin{align} \label{eq:Psi_Lambda_Theta}
\mathrm{diag}\left(- \tfrac{\mathrm{d}}{\mathrm{d}x}w_-\mathds{I}_6, \tfrac{\mathrm{d}}{\mathrm{d}x} w_+ \mathds{I}_6\right) \Lambda + (w_+ - w_-)\Theta
\end{align}
is negative definite, where where $C_\kappa$, $\Lambda$ and $\Theta$ are defined in \eqref{eq:def_C_kappa}, \eqref{eq:def_Lam} and \eqref{eq:def_The} respectively. Indeed, on the one hand, $(\kappa^2 Q_+(0) - Q_-(0)) = \tfrac{1}{2}(w_+(0) \kappa^2 - w_-(0)\mathds{I}_6) \mathbf{M}$ while $(Q_-(\ell ) - Q_+(\ell )) = \tfrac{1}{2}(w_-(\ell ) - w_+(\ell )) \mathbf{M}$,
and both diagonal matrices are negative semi-definite if and only if \eqref{eq:weights_cond_0_L} holds. On the other hand, the products $QB$ and $B^\intercal Q$ now write as (see \cref{rem:special_form_B})
\begin{align*}
QB
= \begin{bmatrix}
-w_-(\mathcal{B}-\mathcal{B}^\intercal) & -w_-(\mathcal{B}+\mathcal{B}^\intercal) \\ w_+(\mathcal{B}+\mathcal{B}^\intercal) & w_+(\mathcal{B}-\mathcal{B}^\intercal)
\end{bmatrix}, \quad B^\intercal Q = \begin{bmatrix}
w_-(\mathcal{B}-\mathcal{B}^\intercal) & w_+(\mathcal{B}+\mathcal{B}^\intercal) \\ -w_-(\mathcal{B}+\mathcal{B}^\intercal) & -w_+(\mathcal{B}-\mathcal{B}^\intercal)
\end{bmatrix},
\end{align*}
where $\mathcal{B} = \tfrac{1}{2} \mathbf{E}D\mathbf{M}$. Hence, the sum yields
\begin{equation*}
QB + B^\intercal Q = (w_+ - w_-) \begin{bmatrix}
\mathds{O}_6 & \mathcal{B}+\mathcal{B}^\intercal \\ \mathcal{B} + \mathcal{B}^\intercal & \mathds{O}_6
\end{bmatrix},
\end{equation*}
and \eqref{eq:coron_th_struct_cond} holds if and only if \eqref{eq:Psi_Lambda_Theta} is negative definite for any $x \in [0, \ell]$.

\subsubsection*{Step 2: Assumption on the weights}

Since $\Theta$ is indefinite, our strategy is to choose $w_-, w_+$ such that the first term in of \eqref{eq:Psi_Lambda_Theta} is negative definite and sufficiently large (in some sense), in comparison to the second term, for \eqref{eq:Psi_Lambda_Theta} to be negative definite. To make this first term negative definite, we additionally assume that
\begin{align} \label{eq:cond_weights_deriv}
\frac{\mathrm{d}w_-}{\mathrm{d}x}>0, \quad \frac{\mathrm{d}w_+}{\mathrm{d}x} <0, \qquad \text{in } [0,\ell ].
\end{align}
Now, not only are the weights are positive, but also $w_+$ is decreasing while $w_-$ is increasing. 
For such weights, \eqref{eq:weights_cond_0_L} is equivalent to
\begin{align}
&\frac{w_+(0)}{w_-(0)}  \in \left(1 \,, C_\kappa^{-1}\right], \label{eq:cond0_2}\\
&w_- \leq w_+, \quad \text{in }[0, \ell]. \label{eq:condell_2}
\end{align}
We now make use of \cref{lem:neg_Psi} to obtain an explicit condition on the weights and their derivatives that is sufficient for \eqref{eq:Psi_Lambda_Theta} to be negative definite for any $x \in [0,\ell ]$. This lemma yields that if $w_- \leq w_+$ in $[0, \ell]$ and
\begin{align} \label{eq:condition_der_diff_weights}
\min \left\{\left|\tfrac{\mathrm{d}w_-}{\mathrm{d}x}\right|, \left| \tfrac{\mathrm{d}w_+}{\mathrm{d}x} \right| \right\} > (w_+ - w_-) q_m, \quad \text{in }[0,\ell ],
\end{align}
for some $m \in \{1, 2\}$, then the matrix \eqref{eq:Psi_Lambda_Theta} is negative definite for any $x \in [0, \ell]$.

\subsubsection*{Step 3: Existence of the weights}

To finish the proof, one has to find weights fulfilling \eqref{eq:weights_positive} and \eqref{eq:cond_weights_deriv}, as well as \eqref{eq:cond0_2}-\eqref{eq:condell_2}-\eqref{eq:condition_der_diff_weights}.
One can easily find different weights satisfying \eqref{eq:weights_positive}, \eqref{eq:cond_weights_deriv} and \eqref{eq:condell_2}, using straight lines, exponential functions or cotangent functions, for instance.
However, it is not straightforward to find weights also satisfying \eqref{eq:cond0_2} and \eqref{eq:condition_der_diff_weights} for realistic beam parameters respecting the assumptions of the beam model,
especially as some of these parameters are linked to the others. For instance, $\ell, I_2, I_3$ and $a$ are related, and so are $E$ and $G$.

We use \cref{lem:existence_varphi} to obtain such weights without adding any constraint on the beam parameters.
Let $m \in \{1, 2\}$ and define $C_{q_m}>0$ by 
\begin{align} \label{eq:def_c_weights}
C_{q_m} = \max_{x \in [0,\ell ]} q_m(x),
\end{align}
where $q_m$ is defined by \eqref{eq:def_q1_q2}.
By \cref{lem:existence_varphi}, there exists $\varphi \in C^1([0, \ell ])$ such that
\begin{align} \label{eq:hyp_g_Cqi}
\varphi >0, \quad \frac{\mathrm{d}\varphi}{\mathrm{d}x}  >0, \quad \frac{\mathrm{d} \varphi}{\mathrm{d}x} > 2 C_{q_m} \, (\varphi (\ell ) - \varphi ), \qquad \text{in }[0, \ell ].
\end{align}

Then, the weights $w_+,w_- \in C^1([0, \ell ]; \mathbb{R})$ defined by
\begin{align} \label{eq:weights_struct}
w_- = \varphi , \qquad w_+ = 2 \varphi (\ell ) - \varphi ,
\end{align}
satisfy \eqref{eq:weights_positive}, \eqref{eq:cond_weights_deriv}, \eqref{eq:condell_2} and \eqref{eq:condition_der_diff_weights}.
Indeed, both weights are positive since $w_- = \varphi  >0$ and $w_+ > 2\varphi (\ell ) - \varphi (0) > 0$, with monoticity $\tfrac{\mathrm{d} w_-}{\mathrm{d}x} = \tfrac{\mathrm{d}\varphi}{\mathrm{d}x} >0$ and $\tfrac{\mathrm{d}w_+}{\mathrm{d}x} = -\tfrac{\mathrm{d}\varphi}{\mathrm{d}x} <0$. Since $\varphi $ satisfies $\tfrac{\mathrm{d}}{\mathrm{d}x} \varphi (x)  >  2 q_m(x) \, (\varphi (\ell ) - \varphi (x))$ for all  $x \in [0, \ell ]$, we deduce that \eqref{eq:condition_der_diff_weights} holds, as $\min \big\{\big|\tfrac{\mathrm{d}w_-}{\mathrm{d}x}(x)\big|, \big|\tfrac{\mathrm{d}w_+}{\mathrm{d}x}(x) \big| \big\} = \tfrac{\mathrm{d}\varphi}{\mathrm{d}x}(x)$ and $(w_+ - w_-)(x) = 2(\varphi (\ell ) - \varphi (x))$.

Furthermore, \eqref{eq:cond0_2} is also fulfilled if $\varphi $ additionally satisfies
\begin{align} \label{eq:hyp_g_0}
\varphi (\ell ) \in \left[\varphi (0) \,, \frac{1+C_\kappa^{-1}}{2}\varphi (0) \right].
\end{align}
This follows from rewriting condition \eqref{eq:cond0_2} using that $w_-(0) = \varphi (0)$ and $w_+(0) = 2\varphi (\ell ) - \varphi (0)$.
This concludes the proof.
\end{proof}

\subsection{Additional comments} Let us make some comments.

\paragraph{Feedback parameters}
In \eqref{eq:cond0_2}, we observe that $\kappa$ determines how different from one another the weights are allowed to be at $x=0$: if $\kappa$ is closer to the null matrix, then the weights are less constrained. 
Hence, for fixed beam parameters $(a, \rho, E, G, I_2, I_3, \{k_i\}_{i=1}^3, \ell)$, the choice of the feedback parameters $\mu_1, \mu_2>0$ influences this constraint. Both $\mu_1$ and $\mu_2$ must be nonzero, as otherwise $C_\kappa = 1$ and the interval $(1 \,, C_\kappa^{-1}]$ is empty.
Furthermore, one can show that the smallest $C_\kappa$ is obtained for
\begin{align} \label{eq:feedbackPara_leastConstrain}
\mu_1 = \sqrt{\left(\min_{1\leq i \leq 3}b_i \right)\left(\max_{1\leq i \leq 3} b_i \right)}, \qquad \mu_2 = \sqrt{\left(\min_{4\leq i \leq 6} b_i \right)\left(\max_{4\leq i \leq 6} b_i \right)},
\end{align}
where $\{b_i\}_{i=1}^6$ are the diagonal entries of $\mathbf{M}D$.

\paragraph{Form of the weights}
\begin{figure}[h]
\centering
\includegraphics[scale = 0.55]{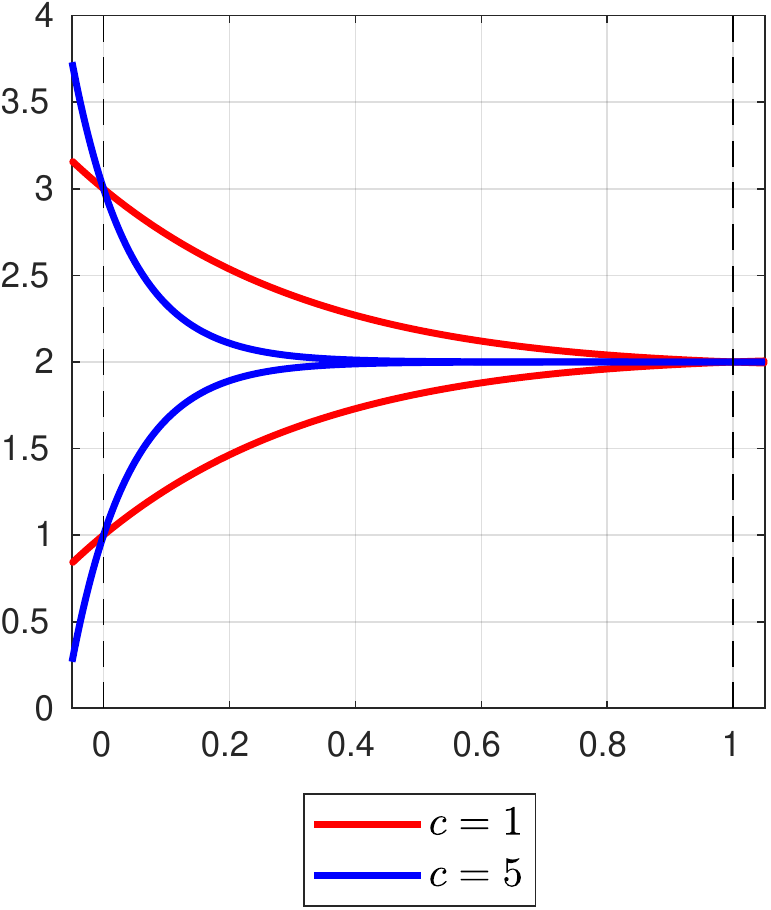}%
\hspace{0.5cm}\includegraphics[scale=0.55]{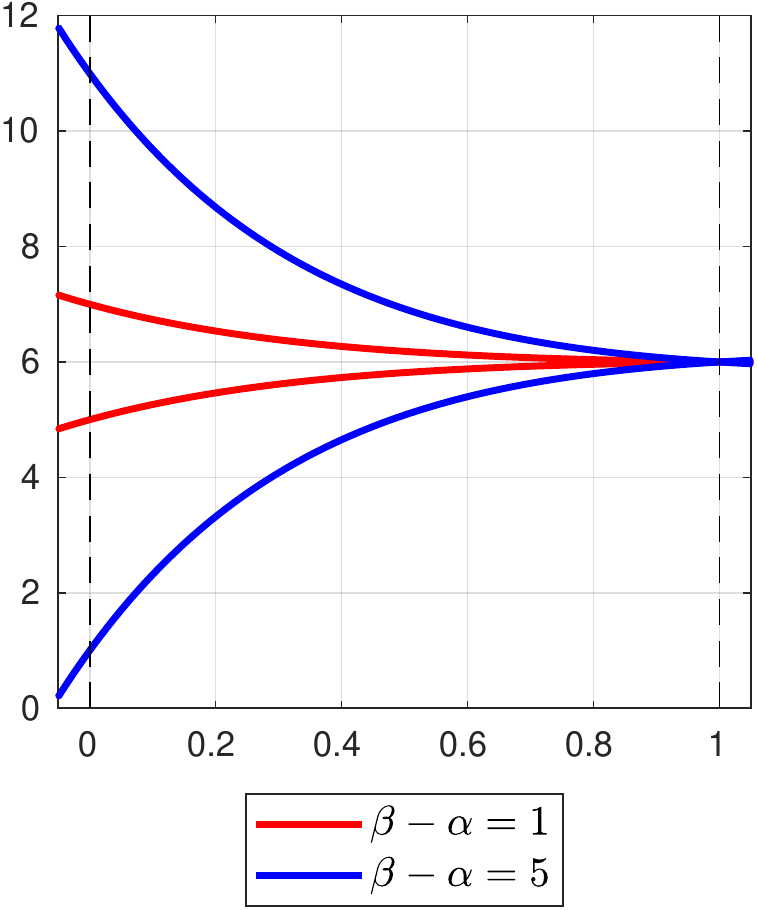}
\caption{The weights \eqref{eq:def_weights_g} for varying $c$ and $\beta-\alpha=1$ (left), for varying $\beta-\alpha$ and $c=1$ (right). Here, $\ell  = \alpha = 1$, the upper decreasing curve is $w_+$ and the lower increasing curve is $w_-$.}
\label{fig:g_weights}
\end{figure}
Let $\alpha>0$ and $c = C_{q_m}$ for $m \in \{1, 2\}$. The function $\varphi $ fulfilling \eqref{eq:hyp_g_Cqi}-\eqref{eq:hyp_g_0} can be chosen as $\varphi (x) = \beta - e^{- 2c x}\left(1 - \frac{x}{\ell }\right)(\beta-\alpha)$ for some $\beta$ belonging to $\left(\alpha \,, \frac{1}{2}(1+C_\kappa^{-1})\alpha \right]$. The corresponding weights \eqref{eq:weights_struct} are (see \cref{fig:g_weights})
\begin{align}\label{eq:def_weights_g}
\Scale[0.95]{
\begin{aligned}
w_-(x) = \beta - e^{- 2c x}\left(1 - \frac{x}{\ell }\right)(\beta-\alpha), \quad w_+(x) = \beta + e^{- 2c x}\left(1 - \frac{x}{\ell }\right)(\beta-\alpha).
\end{aligned}}
\end{align}

\paragraph{Initially straight beam}
In the particular case of a straight and untwisted beam with centerline $p(x) = e_1 x$ before deformation (see \cref{sec:unknowns_coef}), we can compute the constants $C_{q_1}$ and $C_{q_2}$ defined by \eqref{eq:def_c_weights}. They are
\begin{align*}
C_{q_1} &= \max \left\{1, \ a\sqrt{k_3 G} (I_2 \sqrt{E})^{-1}, \ a \sqrt{k_2 G} (I_3 \sqrt{E})^{-1} \right\},\\
C_{q_2} &= \max \{\lambda_8, \lambda_9\} \left(\min \left\{\lambda_7, \ \lambda_8, \ \lambda_9, \ a^{-1}k_1(I_2+I_3)\lambda_{10}, \ a^{-1}I_2 \lambda_7, \ a^{-1}I_3 \lambda_7 \right\}\right)^{-1}.
\end{align*}

\section{Proof of \cref{thm:fromIGEBtoGEB}}
\label{sec:fromIGEBtoGEB}

Finally, making use of the first main result \cref{thm:stabilization}, we want to prove the existence of a unique solution to the GEB model \eqref{eq:GEB} and some stability properties of this solution. 
We will use the notion of \emph{quaternion} (see \cite{chou1992} and the references therein).
A quaternion is a pair of real value $q_0 \in \mathbb{R}$ and vectorial value $q \in \mathbb{R}^3$, that we denote here as the vector $\mathbf{q} =(q_0, q^\intercal)^\intercal$.
A rotation matrix $\mathbf{R} \in \mathrm{SO}(3)$ is said to be parametrized by the quaternion $\mathbf{q} \in \mathbb{R}^4$, if $|\mathbf{q}| = 1$ and 
\begin{align} \label{eq:rot_param_quat}
\mathbf{R} = (q_0^2 - \langle q \,, q \rangle)\mathds{I}_3 + 2 q q^\intercal + 2 q_0 \widehat{q}.
\end{align}
When computing a quaternion from the rotation matrix there is a sign ambiguity as both $\mathbf{q}$ and its opposite $-\mathbf{q}$ represent the same rotation matrix.
We will say that the map $\mathbf{R}\colon [0, \ell]\times [0, T] \rightarrow \mathrm{SO}(3)$ is parametrized by the quaternion-valued function $\mathbf{q}\colon [0, \ell]\times [0, T] \rightarrow \mathbb{R}^4$, if $|\mathbf{q}| \equiv 1$ and \eqref{eq:rot_param_quat} is fulfilled for all $(x,t) \in [0, \ell]\times [0, T]$.

We start by giving two lemmas of use in the proof of \cref{thm:fromIGEBtoGEB}.
The first lemma will allow us to rewrite a linear PDE whose unknown $\mathbf{R}$ has values in $\mathrm{SO}(3)$, as another linear PDE whose unknown state is the quaternion-valued map $\mathbf{q}$ which parametrizes $\mathbf{R}$. Let us introduce the function $\mathcal{U}$ defined by
\begin{align*} 
\mathcal{U}(v) = \frac{1}{2} \begin{bmatrix}
0 & - v^\intercal \\
v & \widehat{v}
\end{bmatrix}, \qquad \text{for all }v \in \mathbb{R}^3. 
\end{align*}

\begin{lemma} \label{lem:quat_rot_ODE}
Let $f \in C^1([0, \ell]\times[0, +\infty); \mathbb{R}^3)$ and let $z$ represent either the variable $x$ or the variable $t$.
The function $\mathbf{q} \in C^1([0, \ell]\times[0, +\infty); \mathbb{R}^4)$ fulfills both $|\mathbf{q}| \equiv 1$ and
\begin{align} \label{eq:zODE_quat_U}
\partial_z \mathbf{q} (x, t)= \mathcal{U}(f(x, t)) \mathbf{q}(x, t), \qquad \text{for all } (x,t) \in [0, \ell]\times[0, +\infty),
\end{align}
if and only if the map $\mathbf{R} \in C^1([0, \ell]\times[0, +\infty); \mathrm{SO}(3))$ parametrized by $\mathbf{q}$ fulfills
\begin{align*} 
\partial_z \mathbf{R}(x, t) = \mathbf{R}(x, t) \widehat{f(x, t)}, \qquad \text{for all } (x,t) \in [0, \ell]\times[0, +\infty).
\end{align*}
\end{lemma}


\begin{remark} \label{rem:constant_quat}
In itself, \eqref{eq:zODE_quat_U} implies that $\partial_z (|\mathbf{q}|^2) \equiv 0$, since it yields $\partial_z (|\mathbf{q}|^2) = 2 \langle \mathbf{q} \,, \mathcal{U}(f) \mathbf{q} \rangle$ and straightforward computations yield that the right-hand side is null.
\end{remark}

By definition of the quaternion product $\circ$ \hspace{0.001cm}, the equation $\partial_z \mathbf{q} = \mathcal{U}(f) \mathbf{q}$ is just is an equivalent way of writing $\partial_z \mathbf{q} = \frac{1}{2} \mathbf{q} \circ \mathbf{f}$, where $\mathbf{f} = (0, f^\intercal)^\intercal$.
The proof of \cref{lem:quat_rot_ODE}, omitted here, rests on extensive but elementary computations involving the relationship \eqref{eq:rot_param_quat}, the definition of the quaternion product and properties of the cross product. The second lemma, given below, yields the existence of a unique solution to an overdetermined system of two first order PDE.

\begin{lemma} \label{lem:quat_overter_syst}
Let $A, B \in C^1([0, \ell]\times [0, +\infty); \mathbb{R}^{n \times n})$ be such that the compatibility condition $A B - B A + (\partial_x A) - (\partial_t B) = 0$ holds in $[0, \ell]\times [0, + \infty)$.
Then,
\begin{align*} 
\begin{cases}
\partial_t y(x, t) = A(x, t) y(x,t) & \text{in }[0, \ell]\times [0, + \infty)\\
\partial_x y(x, t) = B(x, t) y(x,t) & \text{in }[0, \ell]\times [0, + \infty)\\
y(\ell, 0) = y_\mathrm{in}.
\end{cases}
\end{align*}
admits a unique solution $y \in C^1([0, \ell]\times[0, + \infty); \mathbb{R}^{n})$, for any given $y_\mathrm{in} \in \mathbb{R}^n$.
\end{lemma}

The proof consists in considering the solution $y$ to $\partial_t y = A y$ in $[0, \ell]\times [0, + \infty)$ with $y(x, 0) = w(x)$, where $w$ is the solution to $\frac{\mathrm{d}}{\mathrm{d}x} w = B(\cdot, 0) w$ in $[0, \ell]$ with $w(\ell) = y_\mathrm{in}$, and showing that $y$ solves in fact also $\partial_x y = B y$ in $[0, \ell]\times [0, + \infty)$ by using the compatibility condition. We may now prove the second main result.

\begin{proof}[Proof of \cref{thm:fromIGEBtoGEB}]
Let $y \in C^0([0, +\infty); \mathbf{H}^2(0, \ell))$ be the unique solution to \eqref{eq:intrinsic_original} with initial datum $y^0$ fulfilling the assumptions of \cref{thm:fromIGEBtoGEB}. We will also use the notation $y = (\mathbf{y}_1^\intercal, \mathbf{y}_2^\intercal, \mathbf{y}_3^\intercal,\mathbf{y}_4^\intercal)^\intercal$, where $\mathbf{y}_i \colon [0, \ell]\times [0, +\infty)\rightarrow \mathbb{R}^{3}$ for $1 \leq i \leq 4$.

As explained in \cref{rem:extensions} \ref{rem:reg_sol} this solution $y$ belongs to $C^1([0, \ell]\times [0, +\infty); \mathbb{R}^{d})$ and there exists $\alpha, \bar{\eta}>0$ depending only on $\varepsilon$ such that $\|y\|_{C^{1}([0, \ell] \times [0, +\infty); \mathbb{R}^d)} \leq \bar{\eta} e^{- \alpha t} \|y^0\|_{\mathbf{H}^2(0, \ell)}$.
The proof is divided in four steps. 

\subsubsection*{Step 1: Compatibility conditions}
In this step, we point out compatibility conditions useful latter on, which come from the governing equations fulfilled by $y$ and from the relationship between the initial data of \eqref{eq:GEB} and \eqref{eq:intrinsic_original}.
The last six governing equations of \eqref{eq:intrinsic_original}, write equivalently as
\begin{align}
&\partial_t \mathbf{y}_3 - \partial_x \mathbf{y}_1 - (\widehat{\mathbf{y}}_4 + \widehat{\Upsilon}_c)\mathbf{y}_1 + \widehat{\mathbf{y}}_2 (\mathbf{y}_3 + e_1)=0. \label{eq:compat_p}\\
&\partial_x \mathbf{y}_2 - \partial_t \mathbf{y}_4  = \widehat{\mathbf{y}}_2  (\Upsilon_c + \mathbf{y}_4) \label{eq:compat_R}
\end{align}
Using that $\widehat{(\widehat{u}v)} = \widehat{u}\widehat{v} - \widehat{v}\widehat{u}$ for all $u,v \in \mathbb{R}^3$, one can show that \eqref{eq:compat_R} is equivalent to
\begin{align}\label{eq:compat_R_U}
\mathcal{U}(\mathbf{y}_2) \mathcal{U}(\mathbf{y}_4 + \Upsilon_c) - \mathcal{U}(\mathbf{y}_4 + \Upsilon_c)\mathcal{U}(\mathbf{y}_2)  + \partial_x(\mathcal{U}(\mathbf{y}_2)) - \partial_t(\mathcal{U}(\mathbf{y}_4 + \Upsilon_c)) = 0.
\end{align}
From \eqref{eq:y0}, we know that, for all $(x, t) \in [0, \ell]\times[0,+\infty)$,
\begin{align} \label{eq:compat_ini_R}
\tfrac{\mathrm{d}}{\mathrm{d}x} \mathbf{R}^0(x) &= \mathbf{R}^0(x)(\widehat{\mathbf{y}}_4(x, 0) + \widehat{\Upsilon}_c(x)),\qquad \tfrac{\mathrm{d}}{\mathrm{d}t} h^\mathbf{R} = h^\mathbf{R} \widehat{\mathbf{y}}_2(\ell, t),\\
\label{eq:compat_ini_p}
\tfrac{\mathrm{d}}{\mathrm{d}x}\mathbf{p}^0(x)&=\mathbf{R}^0(x)(\mathbf{y}_3(x, 0) + e_1), 
\qquad \quad \ \ \, \mathbf{y}_1(\ell, \cdot) = 0,
\end{align}
where the second equation in \eqref{eq:compat_ini_R} comes from $\mathbf{y}_2(\ell, \cdot)\equiv 0$ and that $h^\mathbf{R}$ is constant.

Recall that $h^\mathbf{p} = \mathbf{p}^0(\ell)$ and $h^\mathbf{R} = \mathbf{R}^0(\ell)$, and that $\mathcal{N}(\mathbf{p}, \mathbf{R}) = y$ also writes as 
\begin{align} \label{eq:transfor_p_R}
\partial_t \mathbf{R} = \mathbf{R} \widehat{\mathbf{y}}_2, \quad \partial_x \mathbf{R} = \mathbf{R} (\widehat{\mathbf{y}}_4 + \widehat{\Upsilon}_c),\quad \partial_t \mathbf{p} = \mathbf{R}\mathbf{y}_1, \quad \partial_x \mathbf{p} = \mathbf{R}(\mathbf{y}_3 + e_1).
\end{align}
Our aim in the next two steps is to show that the transformation $\mathcal{N}$ is bijective from the space $E_1 = \left\{(\mathbf{p}, \mathbf{R}) \in C^2([0, \ell] \times [0, +\infty); \mathbb{R}^3 \times \mathrm{SO}(3)) \colon \eqref{eq:ini_bound_R}, \eqref{eq:ini_bound_p} \right\}$ onto the space $E_2 = \left\{y \in C^1([0, \ell] \times [0, +\infty); \mathbb{R}^d) \colon \eqref{eq:compat_p}, \eqref{eq:compat_R}, \eqref{eq:compat_ini_R}, \eqref{eq:compat_ini_p} \right\}$, where
\begin{align} \label{eq:ini_bound_R}
\mathbf{R}(x, 0) = \mathbf{R}^0(x), \quad \mathbf{R}(\ell, t) = h^\mathbf{R}, \qquad \text{for all }(x, t) \in [0, \ell]\times[0, T],\\
\label{eq:ini_bound_p}
\mathbf{p}(x, 0) = \mathbf{p}^0(x), \quad \  \mathbf{p}(\ell, t) = h^\mathbf{p}, \qquad \ \text{for all }(x, t) \in [0, \ell]\times[0, T].
\end{align}

\subsubsection*{Step 2: Rotation matrix} 
In this step, we show that there exists a unique $\mathbf{R} \in C^1([0, \ell] \times [0, +\infty); \mathrm{SO}(3))$ fulfilling \eqref{eq:ini_bound_R} and the first two equations in \eqref{eq:transfor_p_R}. Let us denote $\mathbf{R}_\mathrm{in} = \mathbf{R}^0(\ell) = h^\mathbf{R}$, and let $\mathbf{q}_\mathrm{in} \in \mathbb{R}^4$ be a quaternion which parametrizes $\mathbf{R}_\mathrm{in}$. Since \eqref{eq:compat_ini_R} holds, imposing the condition $\mathbf{R}(\ell, 0) = \mathbf{R}_\mathrm{in}$ is equivalent to imposing \eqref{eq:ini_bound_R}.
Hence, we will look for the solution $\mathbf{R}$ to
\begin{align} \label{eq:R_overdet_in}
\begin{cases}
\partial_t \mathbf{R} = \mathbf{R} \widehat{\mathbf{y}}_2 &\text{in }[0, \ell] \times [0,+\infty)\\
\partial_x \mathbf{R} = \mathbf{R} (\widehat{\mathbf{y}}_4 + \widehat{\Upsilon}_c)&\text{in }[0, \ell] \times [0,+\infty)\\
\mathbf{R}(\ell, 0) = \mathbf{R}_\mathrm{in}.
\end{cases}
\end{align}
If $\mathbf{R} \in C^1([0, \ell]\times[0, +\infty); \mathrm{SO}(3))$ is parametrized by $\mathbf{q} \in C^1([0, \ell]\times[0, +\infty); \mathbb{R}^4)$, then
$\mathbf{R}$ is solution to \eqref{eq:R_overdet_in} if and only if either $\mathbf{q}$ or $-\mathbf{q}$ is solution to 
\begin{align} \label{eq:overdeter_quat_U}
\begin{cases}
\partial_t \mathbf{q} = \mathcal{U}(\mathbf{y}_2) \mathbf{q} &\text{in }[0, \ell] \times [0,+\infty)\\
\partial_x \mathbf{q} = \mathcal{U} (\mathbf{y}_4 + \Upsilon_c) \mathbf{q} &\text{in }[0, \ell] \times [0, +\infty)\\
\mathbf{q}(\ell, 0) = \mathbf{q}_{\mathrm{in}}.
\end{cases}
\end{align}
Indeed, $\mathbf{R}(\ell, 0) = \mathbf{R}_\mathrm{in}$ is equivalent to either $\mathbf{q}(\ell, 0) = \mathbf{q}_{\mathrm{in}}$ or $\mathbf{q}(\ell, 0) = - \mathbf{q}_{\mathrm{in}}$ and, by \cref{lem:quat_rot_ODE}, the governing equations of  \eqref{eq:R_overdet_in} are equivalent to those of \eqref{eq:overdeter_quat_U}.

\cref{lem:quat_overter_syst} and \eqref{eq:compat_R_U} yield the existence of a unique solution $\mathbf{q} \in C^1([0, \ell]\times[0, +\infty); \mathbb{R}^4)$ to \eqref{eq:overdeter_quat_U}.
Moreover, $|\mathbf{q}|\equiv 1$, since $\partial_t(|\mathbf{q}|^2) \equiv 0$ and $\partial_x (|\mathbf{q}|^2) \equiv 0$ by \cref{rem:constant_quat}, and $|\mathbf{q}_{\mathrm{in}}|=1$.
Hence, the map $\mathbf{R} \in C^1([0, \ell]\times[0, +\infty); \mathrm{SO}(3))$ parametrized by this solution $\mathbf{q}$ is the unique solution to \eqref{eq:R_overdet_in}. Indeed, assume that there are two solutions $\mathbf{R}_1,\mathbf{R}_2$ to \eqref{eq:R_overdet_in}, and $\mathbf{q}_1, \mathbf{q}_2$ are respective corresponding quaternions. Then, up to a minus sign $\mathbf{q}_1,\mathbf{q}_2$ are solutions to \eqref{eq:overdeter_quat_U} and are consequently identically equal up to a minus sign, implying that $\mathbf{R}_1 \equiv \mathbf{R}_2$.

\subsubsection*{Step 3: Position of centerline}
Let $\mathbf{R}$ be the unique solution to \eqref{eq:R_overdet_in} given by the previous step. Our aim in this step is to show the existence of a unique $\mathbf{p} \in C^1([0, \ell]\times [0, + \infty); \mathbb{R}^3)$ fulfilling \eqref{eq:ini_bound_p} and the last two equations in \eqref{eq:transfor_p_R}, i.e.
\begin{align} \label{eq:syst_p}
\begin{cases}
\partial_t \mathbf{p} = \mathbf{R} \mathbf{y}_1 &\text{in }[0, \ell] \times [0,  +\infty)\\
\partial_x \mathbf{p} = \mathbf{R} (\mathbf{y}_3 + e_1) &\text{in }[0, \ell] \times [0,  +\infty)\\
\mathbf{p}(x, 0) = \mathbf{p}^0(x) &\text{for }x \in [0, \ell] \\
\mathbf{p}(\ell, t) = h^\mathbf{p} &\text{for }t \in [0, +\infty).
\end{cases}
\end{align}
We write the governing equations in integral form. Then, $\mathbf{p}_1$ fulfills $\partial_t \mathbf{p}_1 = \mathbf{R} \mathbf{y}_1$ and $\mathbf{p}_1(\cdot, 0) = \mathbf{p}^0$ if and only if $\mathbf{p}_1(x, t) = \mathbf{p}^0(x) + \int_0^t (\mathbf{R} \mathbf{y}_1)(x, s)ds$, which also writes as
\begin{align*}
\mathbf{p}_1(x, t) &= \mathbf{p}^0(\ell) - \int_x^\ell \frac{\mathrm{d}\mathbf{p}^0}{\mathrm{d}x} (\xi)d\xi + \int_0^t (\mathbf{R} \mathbf{y}_1)(\ell, \tau)d\tau - \int_0^t \int_x^\ell \partial_x(\mathbf{R} \mathbf{y}_1)(\xi, \tau)d\xi d\tau.
\end{align*}
Similarly, for the second system, $\mathbf{p}_2$ fulfills $\partial_x \mathbf{p}_2 = \mathbf{R} (\mathbf{y}_3 + e_1)$ and $\mathbf{p}_2(\ell,  \cdot) = h^\mathbf{p}$ if and only if $\mathbf{p}_2(x, t) =  h^\mathbf{p} - \int_x^\ell (\mathbf{R} (\mathbf{y}_3 + e_1)) (\xi, t) d \xi$, which also writes as
\begin{align*}
\mathbf{p}_2(x, t) &= h^\mathbf{p} - \int_x^\ell (\mathbf{R} (\mathbf{y}_3 + e_1)) (\xi, 0) d \xi 
- \int_x^\ell \int_0^t \partial_t(\mathbf{R} (\mathbf{y}_3 + e_1))(\xi, \tau)d\tau d\xi. 
\end{align*}
Since $\mathbf{R}$ solves \eqref{eq:overdeter_quat_U} and has values in $\mathrm{SO}(3)$, the compatibility condition \eqref{eq:compat_p} is equivalent to $\partial_x(\mathbf{R} \mathbf{y}_1) = \partial_t(\mathbf{R} (\mathbf{y}_3 + e_1))$. This observation, in addition to \eqref{eq:compat_ini_p}, imply that $\mathbf{p}_1 \equiv \mathbf{p}_2$ is the unique solution to \eqref{eq:syst_p}.

The functions $\mathbf{p}, \mathbf{R}$ provided by the two previous steps are in fact of regularity $C^2$ in $[0, \ell]\times[0, +\infty)$.  Indeed, $\partial_t \mathbf{R}, \partial_x \mathbf{R} \in C^1([0, \ell]\times[0, +\infty); \mathbb{R}^{3\times 3})$ and $\partial_t \mathbf{p}, \partial_x \mathbf{p} \in C^1([0, \ell]\times[0, +\infty); \mathbb{R}^{3})$, since $\mathbf{p},\mathbf{R}, \Upsilon_c, \{\mathbf{y}_i\}_{i=1}^4$ are $C^1$ with respect to their arguments.

\subsubsection*{Step 4: Solution to \eqref{eq:GEB}}
We have found $(\mathbf{p}, \mathbf{R}) \in  E_1$ such that $\mathcal{N}(\mathbf{p}, \mathbf{R}) = y$. In this step we show that it is solution to \eqref{eq:GEB}, and we that the solution in $C^2([0, \ell] \times [0, +\infty); \mathbb{R}^3 \times \mathrm{SO}(3))$ to \eqref{eq:GEB} is unique.
We now use that $y$ fulfills the initial and boundary conditions, and first six governing equations of \eqref{eq:intrinsic_original}.
Indeed, in these six governing equations, we replace $\mathbf{y}_i$ for $1\leq i \leq 4$ by their expressions in terms of $\mathbf{p}, \mathbf{R}$ and $V, W, \Gamma, \Upsilon$ (defined in \eqref{eq:def_VWGU}). After some computations, using properties of the vector product, 
we obtain the governing equations of \eqref{eq:GEB}.
The boundary conditions at $x=\ell$ are recovered by using those of \eqref{eq:intrinsic_original} together with \eqref{eq:def_VWGU} and $\mathcal{N}(\mathbf{p}, \mathbf{R}) = y$. The remaining initial conditions
are retrieved from \eqref{eq:y0} together with \eqref{eq:def_VWGU} and $\mathcal{N}(\mathbf{p}, \mathbf{R}) = y$.
The uniqueness of the solution to \eqref{eq:GEB} results from the uniqueness of the solution to the IGEB model \eqref{eq:intrinsic_original} and from the fact that $\mathcal{N}$ is bijective from $E_1$ onto $E_2$.

The last assertion of the theorem follows from the exponential decay
of $y$ and the fact that $\mathbf{R}$ has values in the set of rotation (hence, unitary) matrices.
\end{proof}

\section{Conclusion and perspectives}
We have studied a freely vibrating beam described by the GEB and IGEB models. Showing first exponential stability for the latter, we deduced existence, uniqueness and some stability properties for the former model with the same boundary feedback control. Let us make some remarks on the Lyapunov functional used in the proof and on the exponential decay, before commenting on possible extensions.

\paragraph{The Lyapunov functional}
To express the energy of the beam and Lyapunov functionals, we may adopt the point of view of either the \emph{physical} system \eqref{eq:intrinsic_original} or the \emph{diagonal} system \eqref{eq:intrinsic_charact}. In \eqref{eq:energy_physical} and \eqref{eq:def_L0}, we have seen that the energy for the physical and diagonal systems is characterized by the constant matrices $Q^\mathcal{P}$ and $Q^\mathcal{D}$, respectively. The Lyapunov functional $\mathcal{L}$ for the diagonal system, which is given in \eqref{eq:form_Lyap}, may also be written in terms of the physical variable, as
\begin{align*}
\bar{\mathcal{L}}(t) = \sum_{j=0}^k \int_0^\ell \left\langle \partial_t^j y(x,t) \,, \bar{Q}(x) \partial_t^j y(x,t) \right\rangle dx
\end{align*}
where $y$ is the unknown state of \eqref{eq:intrinsic_original}. It is interesting to note that the matrix
\begin{align*} 
Q(x) = \mathrm{diag} \Big(\varphi(x) \mathds{I}_6, \, (2\varphi(\ell)-\varphi(x) ) \mathds{I}_6  \Big) Q^\mathcal{D},
\end{align*}
for $\varphi$ as in \cref{lem:existence_varphi}, which was found in the proof of the main result \cref{thm:stabilization}, takes the form $\bar{Q}(x) = L^{\intercal} Q(x) L$, given below, for the physical system
\begin{align*}
\bar{Q}(x) = \varphi(\ell) Q^{P} +  (\varphi(x) - \varphi(\ell)) \begin{bmatrix}
\mathds{O}_6 &  \mathbf{M}D \\
 \mathbf{M}D  & \mathds{O}_6
\end{bmatrix}.
\end{align*}
We see that the "energy matrix" $Q^\mathcal{P}$ is multiplied by a positive constant, and extradiagonal components dependent on $x$ are added. At the boundary $x=\ell$, where the beam is clamped, $Q$ (resp. $\bar{Q}$) is equal to the "energy matrix" $Q^\mathcal{D}$ (resp. $Q^\mathcal{P}$), while it differs at the end $x=0$ at which the feedback control is applied.

\paragraph{Exponential decay}
Following the proof of \cite[Th. 10.2]{bastin2017exponential} for the special case of System \eqref{eq:intrinsic_charact} while making the constants explicit, one can observe that in the case of $H^1$ stabilization (the $H^2$ case being similar), the exponential decay has the form 
\begin{align*}
\alpha = \tfrac{1}{2} C_Q\big(-C_\mathcal{S} - 4 C_Q C_g \delta \big).
\end{align*}
Above, $\delta > 0$ constrains the size of the initial datum in the $C^0([0, \ell]; \mathbb{R}^d)$ norm (or $C^1$ norm in the $H^2$ case). The constants $C_g, C_Q>0$ depend on $g$ (hence, on the beam parameters) and $Q$, and $C_\mathcal{S} < 0$ is the maximum over $[0, \ell ]$ of the largest eigenvalue of $\mathcal{S} = -\tfrac{\mathrm{d}\varphi}{\mathrm{d}x}\Lambda +2(\varphi (\ell )- \varphi )\Theta$, the matrices $\Lambda, \Theta$ being defined in \eqref{eq:def_Lam}-\eqref{eq:def_The}.
It would be valuable to see how the choice of $\mu_1, \mu_2>0$ and the function $\varphi$ (from \cref{lem:existence_varphi}) affect the decay.
We have seen that the feedback parameters influence the choice of $\varphi$ and that the least restricting choice of $\mu_1, \mu_2$ is \eqref{eq:feedbackPara_leastConstrain}. One may also be interested in the impact of the beam parameters, starting with $\ell>0$, on the decay.

\paragraph{Networks}
Beams may also be studied as part of networks to describe flexible structures: see the modelling done in \cite{LLS93modelling} and the simulations of networks of Cosserat elastic rods carried out in \cite{spillmann09}. Different control problems for networks of linear and nonlinear Timoshenko beams have been treated for instance in \cite{kufner18, LLS93control, leugering99}. Our next interest related to this work is the exponential stabilization for a network of IGEB by applying feedback controls at the nodes.

\paragraph{More general beams}
One could consider a more general IGEB model.
If the size or material of the cross sections varies along the centerline then the parameters of the beam, and consequently the model's coefficients, depend on $x$. 
One may also be interested in the IGEB model accounting for any thin beam made of linear-elastic material, as in \cite{artola2019, hodges2003geometrically, palacios2017invariant}.
An advantage of the method presented here, is that it is possible to generalize the result to the above cases as long as the system remains hyperbolic, though, in the latter case of a general IGEB model, one may have to change $Q$ and be mindful of the regularity of the eigenvalues and eigenvectors of $A$.
As mentioned in \cref{rem:extensions} \ref{rem:freely_vibrating}, another perspective is to assume that external forces, such as gravity \cite[eq. (4)]{artola2019} or aerodynamic forces \cite[eq. (12)]{palacios11_intrinsic}, which can be functions of $x$ or $(\mathbf{p}, \mathbf{R})$, are applied on the beam.

\subsection*{Acknowledgments} We thank M. Gugat, B. Gesh\-kovski, A. Wynn, and the Load Control and Aeroelastics Lab team of Imperial College London, for helpful discussions and advice, and the anonymous reviewers who helped improving this work.

\bibliographystyle{plain}
\bibliography{biblio}

\begin{thebibliography}{10}

\bibitem{alabau2007}
F.~Alabau-Boussouira.
\newblock Asymptotic behavior for {T}imoshenko beams subject to a single
  nonlinear feedback control.
\newblock {\em Nonlinear Differ. Equ. Appl.}, 14:643--669, 2007.

\bibitem{tucsnak2000}
K.~Ammari and M.~Tucsnak.
\newblock Stabilization of {B}ernoulli-{E}uler beams by means of a pointwise
  feedback force.
\newblock {\em SIAM J. Control Optim.}, 39(4):1160--1181, 2000.

\bibitem{artola2019}
M.~Artola, A.~Wynn, and R.~Palacios.
\newblock A nonlinear modal-based framework for low computational cost optimal
  control of 3{D} very flexible structures.
\newblock In {\em 18th European Control Conference}, pages 3836--3841, 2019.

\bibitem{Coron_2016_stab}
G.~Bastin and J.-M. Coron.
\newblock {\em Stability and boundary stabilization of 1-{D} hyperbolic
  systems}, volume~88 of {\em Progr. Nonlinear Differential Equations Appl.}
\newblock Birkh\"{a}user/Springer, 2016.

\bibitem{bastin2017exponential}
G.~Bastin and J.-M. Coron.
\newblock Exponential stability of semi-linear one-dimensional balance laws.
\newblock In {\em Feedback stabilization of controlled dynamical systems},
  volume 473 of {\em Lect. Notes Control Inf. Sci.}, pages 265--278. Springer,
  Cham, 2017.

\bibitem{beauchard2011large}
K.~Beauchard and E.~Zuazua.
\newblock Large time asymptotics for partially dissipative hyperbolic systems.
\newblock {\em Arch. Ration. Mech. Anal.}, 199(1):177--227, 2011.

\bibitem{bhatia_matrix}
R.~Bhatia.
\newblock {\em Matrix analysis}, volume 169 of {\em Grad. Texts in Math.}
\newblock Springer-Verlag, New York, 1997.

\bibitem{chou1992}
J.~C.~K. Chou.
\newblock Quaternion kinematic and dynamic differential equations.
\newblock {\em IEEE Trans. Robot. Autom.}, 8(1):53--64, 1992.

\bibitem{do18}
K.D. Do.
\newblock Stabilization of exact nonlinear {T}imoshenko beams in space by
  boundary feedback.
\newblock {\em J. Sound Vib.}, 422:278 -- 299, 2018.

\bibitem{grazioso2018robot}
S.~Grazioso, G.~Di~Gironimo, and B.~Siciliano.
\newblock A geometrically exact model for soft continuum robots: The finite
  element deformation space formulation.
\newblock {\em Soft robotics}, 0(0), 2018.

\bibitem{gugat2018}
M.~Gugat, V.~Perrollaz, and L.~Rosier.
\newblock Boundary stabilization of quasilinear hyperbolic systems of balance
  laws: exponential decay for small source terms.
\newblock {\em J. Evol. Equ.}, 18(3):1471--1500, 2018.

\bibitem{guo2004}
Faming Guo and Falun Huang.
\newblock Boundary feedback stabilization of the undamped {E}uler-{B}ernoulli
  beam with both ends free.
\newblock {\em SIAM J. Control Optim.}, 43(1):341--356, 2004.

\bibitem{hayat2018exponential}
A.~Hayat.
\newblock {Exponential stability of general 1-D quasilinear systems with source
  terms for the $C^1$ norm under boundary conditions}.
\newblock working paper or preprint, October 2018.

\bibitem{hegarty12}
G.~Hegarty and S.~Taylor.
\newblock Classical solutions of nonlinear beam equations: existence and
  stabilization.
\newblock {\em SIAM J. Control Optim.}, 50(2):703--719, 2012.

\bibitem{hodges1990}
D.~H. Hodges.
\newblock A mixed variational formulation based on exact intrinsic equations
  for dynamics of moving beams.
\newblock {\em Int. J. Solids Struct.}, 26(11):1253--1273, 1990.

\bibitem{hodges2003geometrically}
D.~H. Hodges.
\newblock Geometrically exact, intrinsic theory for dynamics of curved and
  twisted anisotropic beams.
\newblock {\em AIAA Journal}, 41(6):1131--1137, 2003.

\bibitem{horn2012matrix}
R.~A. Horn and C.~R. Johnson.
\newblock {\em Matrix analysis}.
\newblock CUP, Cambridge, second edition, 2013.

\bibitem{kimrenardy87}
J.~U. Kim and Y.~Renardy.
\newblock Boundary control of the {T}imoshenko beam.
\newblock {\em SIAM J. Control Optim.}, 25(6):1417--1429, 1987.

\bibitem{Kmit_classical_nonlin}
I.~Kmit.
\newblock Classical solvability of nonlinear initial-boundary problems for
  first-order hyperbolic systems.
\newblock {\em Int. J. Dyn. Syst. Differ. Equ.}, 1(3):191--195, 2008.

\bibitem{kufner18}
T.~Kufner, G.~Leugering, J.~Semmler, M.~Stingl, and C.~Strohmeyer.
\newblock Simulation and structural optimization of 3d {T}imoshenko beam
  networks based on fully analytic network solutions.
\newblock {\em ESAIM Math. Model. Numer. Anal.}, 52(6):2409--2431, 2018.

\bibitem{LLS93control}
J.~E. Lagnese, G.~Leugering, and E.~J. P.~G. Schmidt.
\newblock Control of planar networks of {T}imoshenko beams.
\newblock {\em SIAM J. Control Optim.}, 31(3):780--811, 1993.

\bibitem{LLS93modelling}
J.~E. Lagnese, G.~Leugering, and E.~J. P.~G. Schmidt.
\newblock Modelling of dynamic networks of thin thermoelastic beams.
\newblock {\em Math. Methods Appl. Sci.}, 16(5):327--358, 1993.

\bibitem{leugering99}
G.~Leugering.
\newblock Dynamic domain decomposition of optimal control problems for networks
  of strings and {T}imoshenko beams.
\newblock {\em SIAM J. Control Optim.}, 37(6):1649--1675, 1999.

\bibitem{Li_1994_global}
T.~Li.
\newblock {\em Global classical solutions for quasilinear hyperbolic systems},
  volume~32 of {\em Res. Appl. Math.}
\newblock Masson, Paris; John Wiley \& Sons, Ltd., Chichester, 1994.

\bibitem{li2010controllability}
T.~Li.
\newblock {\em Controllability and observability for quasilinear hyperbolic
  systems}, volume~3 of {\em AIMS Ser. Appl. Math.}
\newblock Am. Inst. Math. Sci., Springfield, MO; Higher Education Press,
  Beijing, 2010.

\bibitem{morgul91}
\"{O}. Morg\"{u}l.
\newblock Boundary control of a {T}imoshenko beam attached to a rigid body:
  planar motion.
\newblock {\em Internat. J. Control}, 54(4):763--791, 1991.

\bibitem{palacios2017invariant}
R.~Palacios.
\newblock {\em Invariant manifolds in beam dynamics: free vibrations and
  nonlinear normal modes}, pages 1--8.
\newblock Springer Berlin Heidelberg, 2017.

\bibitem{palacios11_intrinsic}
R.~Palacios and B.~Epureanu.
\newblock {\em An Intrinsic Description of the Nonlinear Aeroelasticity of Very
  Flexible Wings}.
\newblock 2011.

\bibitem{palacios10aircraft}
R.~Palacios, J.~Murua, and R.~Cook.
\newblock Structural and aerodynamic models in nonlinear flight dynamics of
  very flexible aircraft.
\newblock {\em AIAA Journal}, 48(11):2648--2659, 2010.

\bibitem{quinnrussel78}
J.~P. Quinn and D.~L. Russell.
\newblock Asymptotic stability and energy decay rates for solutions of
  hyperbolic equations with boundary damping.
\newblock {\em Proc. Roy. Soc. Edinburgh Sect. A}, 77(1-2):97--127, 1977.

\bibitem{reissner1981finite}
E.~Reissner.
\newblock On finite deformations of space-curved beams.
\newblock {\em Zeitschrift f{\"u}r angewandte Mathematik und Physik ZAMP},
  32(6):734--744, 1981.

\bibitem{simo1985finite}
J.C. Simo.
\newblock A finite strain beam formulation. {T}he three-dimensional dynamic
  problem. {P}art {I}.
\newblock {\em Comput. Methods in Appl. Mech. and Engrg.}, 49(1):55 -- 70,
  1985.

\bibitem{krstic2009}
A.~{Smyshlyaev}, B.~{Guo}, and M.~{Krstic}.
\newblock Arbitrary decay rate for {E}uler-{B}ernoulli beam by backstepping
  boundary feedback.
\newblock {\em IEEE Trans. Automat. Contr.}, 54(5):1134--1140, 2009.

\bibitem{spillmann09}
J.~Spillmann and M.~Teschner.
\newblock Cosserat nets.
\newblock {\em IEEE Trans. Vis. Comput. Graph.}, 15:325--338, 2009.

\bibitem{strohm_dissert}
C.~Strohmeyer.
\newblock {\em Networks of nonlinear thin structures - theory and
  applications}.
\newblock PhD thesis, FAU University Press, 2018.

\bibitem{tartar1981some}
L.~C. Tartar.
\newblock Some existence theorems for semilinear hyperbolic systems in one
  space variable.
\newblock Technical report, Wisconsin Univ-Madison Mathematics Research Center,
  1981.

\bibitem{wang2014windturbine}
L.~Wang, X.~Liu, N.~Renevier, M.~Stables, and G.~M. Hall.
\newblock Nonlinear aeroelastic modelling for wind turbine blades based on
  blade element momentum theory and geometrically exact beam theory.
\newblock {\em Energy}, 76:487 -- 501, 2014.

\bibitem{wang2006exact}
Z.~Wang.
\newblock Exact controllability for nonautonomous first order quasilinear
  hyperbolic systems.
\newblock {\em Chinese Ann. Math. Ser. B}, 27(6):643--656, 2006.

\bibitem{weiss99}
H.~Weiss.
\newblock {\em Zur Dynamik geometrisch nichtlinearer Balken}.
\newblock PhD thesis, Technische Universit\"at Chemnitz, 1999.

\bibitem{xu2005}
G.~Q. Xu.
\newblock Boundary feedback exponential stabilization of a {T}imoshenko beam
  with both ends free.
\newblock {\em Internat. J. Control}, 78(4):286--297, 2005.

\end{thebibliography}
\end{document}